\documentclass[11pt]{article}
\usepackage{latexsym,amsfonts,amssymb,amsmath,amsthm}
\usepackage{graphicx}
\usepackage{url}
\usepackage[usenames,dvipsnames]{color}
\usepackage{ulem,comment}
\usepackage{tcolorbox}
\usepackage[framed]{matlab-prettifier}
\usepackage{relsize,exscale}
\usepackage{bigints}
\usepackage{caption}
\usepackage{subcaption}
\usepackage[ruled,vlined]{algorithm2e} 
\usepackage{color}
\usepackage{array}

\usepackage{float}
\usepackage{lipsum} 
\usepackage{fancyhdr}
\begin{document}

\newtheorem{theorem}{Theorem}[section]
\newtheorem{proposition}[theorem]{Proposition}
\newtheorem{definition}{Definition}[section] 
\newtheorem{example}{Example}[section]
\newtheorem{exercise}{Exercise}[section]
\newtheorem{application}{Application}[section]
\newtheorem{homework}{Homework}
\newtheorem{practice}{Practice}
\newtheorem{corollary}{Corollary}[section]
\newtheorem{lemma}[theorem]{Lemma}
\newtheorem{assumption}[theorem]{Assumption}
\newtheorem{remark}{Remark}[section]
\newtheorem{notation}[theorem]{Notation}
\numberwithin{equation}{section}

\newcommand{\stk}[2]{\stackrel{#1}{#2}}
\newcommand{\dwn}[1]{{\scriptstyle #1}\downarrow}
\newcommand{\upa}[1]{{\scriptstyle #1}\uparrow}
\newcommand{\nea}[1]{{\scriptstyle #1}\nearrow}
\newcommand{\sea}[1]{\searrow {\scriptstyle #1}}
\newcommand{\csti}[3]{(#1+1) (#2)^{1/ (#1+1)} (#1)^{- #1
 / (#1+1)} (#3)^{ #1 / (#1 +1)}}
\newcommand{\RR}[1]{\mathbb{#1}}

\newcommand{\rd}{{\mathbb R^d}}
\newcommand{\ep}{\varepsilon}
\newcommand{\rr}{{\mathbb R}}
\newcommand{\alert}[1]{\fbox{#1}}
\newcommand{\eqd}{\sim}
\def\p{\partial}
\def\R{{\mathbb R}}
\def\N{{\mathbb N}}
\def\Q{{\mathbb Q}}
\def\C{{\mathbb C}}
\def\l{{\langle}}
\def\r{\rangle}
\def\t{\tau}
\def\k{\kappa}
\def\a{\alpha}
\def\la{\lambda}
\def\De{\Delta}
\def\de{\delta}
\def\ga{\gamma}
\def\Ga{\Gamma}
\def\ep{\varepsilon}
\def\eps{\varepsilon}
\def\si{\sigma}
\def\Re {{\rm Re}\,}
\def\Im {{\rm Im}\,}
\def\E{{\mathbb E}}
\def\P{{\mathbb P}}
\def\Z{{\mathbb Z}}
\def\D{{\mathbb D}}
\newcommand{\ceil}[1]{\lceil{#1}\rceil}

\title{An adjoint-free integral feedback method for a parabolic inverse source problem with conditional stability}
\author{
  Sedar Ngoma\footnote{Department of Mathematics, State University of New York, Geneseo, NY 14454 (ngoma@geneseo.edu)}
  }
\date{} 
\maketitle


\begin{abstract}
We study an inverse source problem for a linear non-autonomous parabolic equation with additive source term $f(t)+\zeta(t,x)$, where the unknown component depends only on time and is recovered from an integral observation of the solution. After reducing the problem to an equivalent linear inverse problem, we establish existence and uniqueness of the Tikhonov-regularized solution and derive a first-order optimality condition. We show that the forward operator admits a Volterra representation in time, yielding a weak-norm stability estimate in $H^{-1}(0,T)$ and, under an a priori $H^r(0,T)$ bound on the source, a conditional H\"{o}lder stability estimate in $L^2(0,T)$.

Motivated by this Volterra structure, we introduce an adjoint-free integral feedback method that reconstructs the source using only forward solves. We analyze the feedback iteration by establishing its well-definedness and fixed-point properties, convergence for exact data, and finite-iteration stability with respect to noisy data. We further show that, with an appropriate noise-dependent stopping rule, the method constitutes an iterative regularization scheme. Numerical experiments for smooth and piecewise constant sources, supplemented by temporal regularization and automatic parameter selection, demonstrate accurate and stable reconstructions in the presence of noise.
\end{abstract}

\section{Introduction}

Let $\Omega \subset \mathbb{R}^d$, $d\ge 1$, be a bounded domain with smooth boundary $\partial \Omega$ and measure $|\Omega|>0$, and let $T > 0$. Consider the parabolic partial differential equation (PDE):
\begin{alignat}{2}
\partial_t u - \nabla \cdot (a(t,x) \nabla u) + b(t,x) \cdot \nabla u + c(t,x) u &= f(t) + \zeta(t,x), &&\quad (t,x) \in (0,T)\times \Omega \label{parab1} \\
u(0,x) &= u_{0}(x), &&\quad x \in\Omega \label{ic} \\
a(t,x) \nabla u \cdot \nu &= g(t,x), &&\quad (t,x) \in (0,T)\times\partial\Omega \label{bc},
\end{alignat}
where $\nu$ is the outward unit normal to $\partial\Omega$. We refer to the equation as non-autonomous, since the coefficients of the differential operator depend explicitly on time.

Given the coefficients $a(t,x), b(t,x), c(t,x)$, the source function $f(t)+\zeta(t,x)$, and the initial and boundary conditions, problem~\eqref{parab1}--\eqref{bc} is referred to as a direct or forward problem. It is widely used in various applications in science and technology, mathematics, physics, engineering, and many other fields. It describes the evolution in time of the density of some quantity $u(t,x)$, such as the concentration of a chemical or temperature within the region $\Omega$. The second term on the left-hand side of Equation~\eqref{parab1} describes diffusion, the third describes transport, and the fourth describes creation or depletion~\cite[pp.~372--373]{Evans10}.

On the other hand, given the above coefficients, the source function $\zeta(t,x)$, and the initial and boundary conditions, the problem of determining the pair $(u,f)$ in~\eqref{parab1}--\eqref{bc} is called an inverse source problem. As such, the problem is underdetermined, and an additional condition on the solution $u$ of the direct problem is required in order to recover the source $f$. Several types of additional conditions have been considered in the literature depending on whether the unknown function is space-dependent or time-dependent. These include, but are not limited to, knowledge of the solution at the final time, an integral observation of the solution, observation of the solution at an interior spatial point, and knowledge of the normal derivative at a boundary point.

This work is concerned with the recovery of the time-dependent source function $f(t)$ from the integral constraint
\begin{equation}
\int_{\Omega}u(t,x)\,dx = \mu(t), \quad t \in [0,T]. \label{integ1}
\end{equation}
The problem is therefore as follows: given the component $\zeta(t,x)$ of the source, the initial and boundary conditions, the integral observation $\mu(t)$, and the coefficients $a,b,c$, determine the pair $(u,f)$ satisfying Equations~\eqref{parab1}--\eqref{integ1}.

If $u(t,x)$ represents temperature, then $\mu(t)$ can be interpreted as a measurement of the total thermal energy. While the direct problem consists of determining the temperature distribution $u$ given the source $f$, the inverse problem seeks to recover the source function $f$ that produced the observation $\mu$.

Inverse source problems for parabolic equations arise in a wide range of applications, including heat conduction, diffusion processes, environmental modeling, and epidemiology, where one seeks to reconstruct an unknown forcing term from indirect or incomplete measurements of the system state. In many practical situations, measurements are not available pointwise in space but rather in an averaged or aggregated form, leading naturally to inverse problems with integral observations. Such nonlocal measurements appear, for instance, in thermal imaging, where sensors may record averaged temperature, in population dynamics, where only total population counts are accessible, and in chemical processes, where only bulk quantities can be observed. See, for example,~\cite{cannon1986diffusion, GlotovHamesMeirNgoma2}.

These problems are inherently ill-posed in the sense of Hadamard, as small perturbations in the data may lead to large deviations in the reconstructed source. Consequently, both analytical understanding and appropriate regularization strategies are essential for obtaining stable and physically meaningful solutions.

A substantial body of work has been devoted to inverse source problems for parabolic equations under classical regularity assumptions, particularly in H\"older spaces. In this setting, well-posedness, regularity, and numerical reconstruction methods have been studied extensively; see, for example,~\cite{Cannon84,GlotovHamesMeirNgoma1,HazaneeEtAl13,Isakov2006,Yamamoto09} and the references therein. The reader interested in an inverse scattering problem may consult~\cite{Kedzierawski1993}.

In particular, Ngoma~\cite{ngoma2024well} established the well-posedness of an inverse source problem with an integral observation in parabolic H\"older spaces. The analysis was carried out under the assumption that the zeroth-order coefficient vanishes, $c(t,x)=0$, and that the unknown source depends only on time. It was shown that the inverse problem can be reduced to a Volterra integral equation of the first kind, thereby revealing its ill-posed nature; see also~\cite{DamirchiEtAl19}. Numerical reconstruction was performed using a collocation method combined with Tikhonov regularization and the Morozov discrepancy principle.

Despite these advances, the existing theory is largely confined to classical solutions in H\"older spaces. In many applications, however, the available data may be noisy or of limited regularity, and the coefficients of the governing equation may not possess the smoothness required for classical well-posedness. This motivates the development of a framework based on weak formulations, where solutions are sought in Hilbert spaces and the inverse problem is analyzed through variational methods.

To the best of our knowledge, while inverse source problems for parabolic equations have been studied in weak formulations and with integral-type observations (see, e.g.,~\cite{Hasanov2011,hasanov2014unified,PrilepkoKamyninKostin2017,VanBockstal2025}), these works typically focus either on well-posedness, adjoint-based optimization methods, or specific numerical schemes, and are often restricted to autonomous or time-independent operators.

In contrast, a systematic framework that combines a weak variational formulation for non-autonomous parabolic equations, a detailed analysis of the forward operator and its stability properties, and a mathematically analyzed adjoint-free reconstruction method has not been fully developed.

The goal of this work is to address this gap for a non-autonomous parabolic equation with an additive source term of the form $f(t)+\zeta(t,x)$, where the unknown component depends only on time. The reader interested in the meaning and occurrence of additive and multiplicative sources may consult~\cite{ngoma2024well} and the references therein.

Our contributions are fourfold. First, we formulate the inverse problem as a minimization problem for a Tikhonov functional in $L^2(0,T)$ and prove the existence and uniqueness of a minimizer using tools from the calculus of variations; see, for example,~\cite{EkelandTemam99,EnglHankeNeubauer1996,Lions69}. We further derive a first-order optimality condition through an adjoint equation and obtain a variational characterization of the regularized solution.

Second, we investigate the structure and stability properties of the forward operator. We show that it admits a Volterra representation in time, which yields a weak-norm stability estimate in $H^{-1}(0,T)$ and reveals the smoothing mechanism responsible for the ill-posedness of the inverse problem. By combining this estimate with an a priori $H^r(0,T)$ bound on the source, we derive a conditional H\"older stability estimate in $L^2(0,T)$.

Third, motivated by this Volterra structure, we introduce an adjoint-free integral feedback method for reconstructing the unknown source. For adjusted data $h^\delta$ obtained by removing the known contributions from the observation, the basic feedback iteration takes the form
\begin{equation}\label{eq:intro_feedback}
f^{(k+1)}
=
f^{(k)}
-
\frac{1}{\alpha}
\Phi\bigl(F(f^{(k)})-h^\delta\bigr),
\end{equation}
where $\alpha>0$ is the feedback parameter and $\Phi$ is the Nemytskii (or superposition) operator induced by a nonlinear feedback function $\phi$. In contrast to classical PDE-constrained optimization approaches, the iteration~\eqref{eq:intro_feedback} does not require the computation of the adjoint operator $F^*$ or the solution of an additional adjoint PDE. Instead, the source is updated directly from the data misfit using forward solves only.

Fourth, we provide a mathematical analysis of the integral feedback iteration. Under suitable assumptions on the feedback function, we establish the well-definedness and Lipschitz continuity of the feedback operator and characterize exact solutions as its fixed points. We then derive stability estimates for the residuals and establish convergence of the exact-data iteration under appropriate assumptions. For noisy data, we prove finite-iteration stability and show that a suitable noise-dependent early stopping rule turns the feedback iteration into an iterative regularization method. These results provide a theoretical connection between the proposed reconstruction algorithm and the stability analysis of the inverse problem.

For the numerical implementation, the feedback iteration is supplemented with second-order temporal regularization to further suppress oscillations caused by noisy data. The feedback and regularization parameters are selected automatically using the Morozov discrepancy principle together with a curvature criterion. Numerical experiments with smooth and piecewise constant sources are used to assess the accuracy and robustness of the resulting reconstruction procedure.

The remainder of the paper is organized as follows. In Section~\ref{sec:assumptions_preliminary}, we state the assumptions used throughout the analysis, and in Section~\ref{sec:wellposedness}, we recall the well-posedness of the direct problem under these assumptions. Section~\ref{sec:existence_uniqueness} is devoted to the variational formulation and the existence and uniqueness of the Tikhonov-regularized solution. In Section~\ref{sec:optimality}, we derive the first-order optimality condition and provide a variational characterization of the minimizer, together with additional regularity properties. Section~\ref{sec:cond_stab} analyzes the forward operator and the degree of ill-posedness of the inverse problem. In particular, we establish its Volterra structure, derive a weak-norm stability estimate in $H^{-1}(0,T)$, and obtain a conditional H\"older stability result in $L^2(0,T)$ under suitable a priori bounds. Section~\ref{sec:feedback_analysis} introduces the integral feedback iteration and develops its mathematical analysis, including well-definedness, fixed-point properties, exact-data convergence, finite-iteration stability with respect to noisy data, and iterative regularization by early stopping. Finally, in Section~\ref{sec:numerical_results}, we describe the temporally regularized numerical implementation, the automatic parameter selection strategy, and numerical experiments for both smooth and piecewise constant sources.

\section{Assumptions}\label{sec:assumptions_preliminary}
Throughout this work, we make the following assumptions.
\begin{description}
\item[(A1) Operator coefficients.]
The matrix-valued function $a(t,x)\in\mathbb{R}^{d\times d}$ is symmetric and uniformly elliptic:
there exists $\sigma>0$ such that
\[
a(t,x)\xi\cdot\xi \ge \sigma|\xi|^2
\quad \text{for all } \xi\in\mathbb{R}^d
\text{ and a.e. } (t,x)\in(0,T)\times\Omega.
\]
Moreover, $a \in L^\infty\bigl(0,T;L^\infty(\Omega;\mathbb{R}^{d\times d})\bigr)\cap W^{1,\infty}\bigl(0,T;L^\infty(\Omega;\mathbb{R}^{d\times d})\bigr)$.

The lower-order coefficients satisfy $b,c \in L^\infty\bigl((0,T)\times\Omega\bigr)$.
\item[(A2) Known data.]
The initial condition, distributed source, and boundary data satisfy
\[
u_0\in L^2(\Omega),\qquad
\zeta\in L^2(0,T;L^2(\Omega)),\qquad
g\in L^2(0,T;L^2(\partial\Omega)).
\]
\item[(A3) Unknown source.]
The time-dependent source satisfies $f\in L^2(0,T)$.
\end{description}

\section{Well-posedness of the Direct Problem}\label{sec:wellposedness}
Under assumptions (A1)–(A3), the direct problem is well-posed in the following sense.

\begin{theorem}[Well-posedness of the direct problem]\label{thm:well_posedness}
Under assumptions (A1)–(A3), the direct problem admits a unique weak solution
\[
u_f\in L^2(0,T;H^1(\Omega)), \quad\text{with}\quad 
\partial_t u_f\in L^2(0,T;(H^1(\Omega))^*),
\]
where $(H^1(\Omega))^*$ denotes the dual space of $H^1(\Omega)$.

Moreover, $u_f\in C([0,T];L^2(\Omega))$, and the mapping
$t\mapsto \|u_f(t)\|_{L^2(\Omega)}^2$ is absolutely continuous with
\[
\frac12\frac{d}{dt}\|u_f(t)\|_{L^2(\Omega)}^2
= \langle \partial_t u_f(t),u_f(t)\rangle_{(H^1(\Omega))^*,\,H^1(\Omega)}
\quad \text{for a.e. } t\in(0,T).
\]
\end{theorem}

\begin{proof}
This follows from standard parabolic theory; see, e.g.,~\cite{Evans10,LionsMagenes1972}.
\end{proof}

\section{Existence and Uniqueness of Solutions}\label{sec:existence_uniqueness}
We begin with the following result.
\begin{lemma}\label{lem:linearity}
Let $f\in L^2(0,T)$ and let $u_f$ denote the corresponding weak solution of the direct problem
\begin{alignat}{2}
\partial_t u_f - \nabla \cdot (a(t,x) \nabla u_f) + b(t,x) \cdot \nabla u_f + c(t,x) u_f &= f(t), &&\quad (t,x) \in (0,T)\times \Omega \label{parab2} \\
u_f(0,x) &= 0, &&\quad x \in\Omega \label{ic2} \\
a(t,x) \nabla u_f \cdot \nu &= 0, &&\quad (t,x) \in (0,T)\times\partial\Omega \label{bc2}.
\end{alignat}
Then the map $F: L^2(0,T)\to L^2(0,T)$ defined by
\[
F(f)(t) := \int_\Omega u_f(t,x)\,dx,
\]
is a linear operator.
\end{lemma}

\begin{proof}
Let $f_1,f_2\in L^2(0,T)$ and $\psi,\omega\in\mathbb{R}$. Let $u_{f_1}$ and $u_{f_2}$ be the corresponding solutions to PDE~\eqref{parab2}--\eqref{bc2} with source terms $f_1$ and $f_2$, respectively. 
By linearity of the PDE and uniqueness of weak solutions, we have
\[
u_{\psi f_1 + \omega f_2} = \psi u_{f_1} + \omega u_{f_2}.
\]
Therefore,
\[
F(\psi f_1 + \omega f_2)(t)
= \int_\Omega (\psi u_{f_1} + \omega u_{f_2})(t,x)\,dx
= \psi F(f_1)(t) + \omega F(f_2)(t),
\]
which proves linearity.
\end{proof}

Next, we introduce the adjusted data and the adjusted forward map, which will be central to the proof of the existence and uniqueness of solutions. To this end, assume that assumptions $(A1)-(A3)$ hold and let $\widetilde{F}: L^2(0,T)\to L^2(0,T)$ denote the forward map of system~\eqref{parab1}--\eqref{bc}. Then the spatial integral of the solution at each time is given by
\[
\widetilde{F}(f)(t) := \int_\Omega u(t,x) \, dx = \mu(t),\quad\text{ for all $t\in [0,T]$},
\]
where $u$ solves the PDE~\eqref{parab1}--\eqref{bc}. Since the PDE is linear, we decompose the solution $u$ as
\[
u(t,x) = u_f(t,x) + u_\zeta(t,x) + u_g(t,x) + u_{u_0}(t,x),
\]
where $u_f$ solves the PDE~\eqref{parab2}--\eqref{bc2}, $u_\zeta$ solves the PDE~\eqref{parab1}--\eqref{bc} with source term $\zeta(t,x)$, zero initial, and homogeneous boundary data, $u_g$ solves the same PDE~\eqref{parab1}--\eqref{bc} with zero source, zero initial, but nonzero Neumann boundary data $g(t,x)$, and $u_{u_0}$ solves the same PDE~\eqref{parab1}--\eqref{bc} with zero source, zero Neumann boundary data, but nonzero initial condition $u_0(x)$.

Define the adjusted data
\begin{equation}\label{adjusted_data}
h(t) := \mu(t) - \int_\Omega u_\zeta(t,x)\,dx - \int_\Omega u_g(t,x)\,dx - \int_\Omega u_{u_0}(t,x)\,dx,
\end{equation}
where $\mu(t)$ is given by~\eqref{integ1}. This allows us to reformulate the inverse problem in terms of the linearly adjusted forward map as in Lemma~\ref{lem:linearity}
\begin{equation}
\label{forward_map2}
F(f)(t) := \int_\Omega u_f(x,t)\,dx,
\end{equation}
and the adjusted data $h(t)$ in~\eqref{adjusted_data}, leading to the linear inverse problem:
\begin{equation}\label{new_inverse}
\text{Find $f\in L^2(0,T)$ such that $F(f)(t) = h(t)$}.
\end{equation}

From this point on, we focus exclusively on the analysis of the adjusted forward map $F$ (which we will refer to as the forward map for simplicity) in~\eqref{forward_map2} and the corresponding adjusted data $h$ in~\eqref{adjusted_data}.
The following result shows that the forward map~\eqref{forward_map2} is a bounded operator.
\begin{lemma}\label{lem:bounded}
Assume Assumptions (A1)--(A3) hold. Then, the forward map $F: L^2(0,T)\to L^2(0,T)$ defined in~\eqref{forward_map2} is a bounded operator.
\end{lemma}

\begin{proof}
We first derive the weak formulation. Due to the Neumann boundary conditions, the weak solution $u_f$ of the system~\eqref{parab2}--\eqref{bc2} is sought in the space $L^2(0,T; H^1(\Omega))$ with $\partial_tu_f\in L^2(0,T; (H^1(\Omega))^*)$. To that end, we multiply~\eqref{parab2} by a test function $v\in H^1(\Omega)$, integrate and take $v = u_f(t)$ for almost every $t\in (0,T)$ to get
\begin{equation}\label{weak_form}
\int_\Omega\partial_tu_f\cdot u_f\,dx = \int_\Omega\nabla\cdot (a\nabla u_f)\cdot u_f\,dx - \int_\Omega (b\cdot\nabla u_f)\cdot u_f\,dx - \int_\Omega cu_f^2\,dx + f(t)\int_\Omega u_f\,dx.
\end{equation}
Next, we find an estimate for each term in~\eqref{weak_form}. 

The estimate for the time derivative term follows from Theorem~\ref{thm:well_posedness}: for almost all $t\in (0,T)$,
\begin{equation}\label{time_derivative}
\frac{1}{2}\frac{d}{dt}\|u_f(t)\|^2_{L^2(\Omega)}= \langle \partial_tu_f(t), u_f(t)\rangle_{(H^1)^*,H^1} = \int_\Omega\partial_tu_f\cdot u_f\,dx.
\end{equation}
For the diffusion term, we have thanks to the Divergence Theorem that
\begin{align}\label{diffusion_term}
\int_\Omega\nabla\cdot (a\nabla u_f)\cdot u_f\,dx &= - \int_\Omega a |\nabla u_f|^2\,dx \le - \sigma\|\nabla u_f(t)\|_{L^2(\Omega)}^2.
\end{align}
where we have used the uniform ellipticity in assumption (A1). The boundary term vanishes due to~\eqref{bc2}.

For the convection term, we have the following.
\begin{align}\label{convection_term}
- \int_\Omega (b\cdot\nabla u_f)\cdot u_f\,dx 
&\le \|b\|_{L^{\infty}([0,T]\times\Omega)}\|\nabla u_f(t)\|_{L^2(\Omega)}\|u_f(t)\|_{L^2(\Omega)}\notag\\
&\le \frac{\varepsilon}{2}\|\nabla u_f(t)\|^2_{L^2(\Omega)} + \frac{1}{2\varepsilon}\|b\|^2_{L^{\infty}([0,T]\times\Omega)}\|u_f(t)\|^2_{L^2(\Omega)},
\end{align}
where, in the last inequality we have employed Young's Inequality with $\varepsilon$ to absorb the result into an energy estimate. Here, $\varepsilon>0$ is arbitrary.

For the reaction term, we have
\begin{align}\label{reaction_term}
- \int_\Omega cu_f^2\,dx &\le  \|c\|_{L^{\infty}([0,T]\times\Omega)}\|u_f(t)\|^2_{L^2(\Omega)}.
\end{align}
For the source term, we have
\begin{align}\label{source_term}
f(t)\int_\Omega u_f\,dx &\le\frac{1}{2}\|u_f(t)\|^2_{L^2(\Omega)} + \frac{|\Omega|}{2}|f(t)|^2,
\end{align}
where we have used Young's Inequality in the last quantity. 

Next, substituting~\eqref{time_derivative}--\eqref{source_term} into~\eqref{weak_form} and arranging, we obtain
\begin{align*}
\frac{1}{2}\frac{d}{dt}\|u_f(t)\|^2_{L^2(\Omega)} + \left(\sigma -\frac{\varepsilon}{2}\right)\|\nabla u_f(t)\|_{L^2(\Omega)}^2 &\le\frac{1}{2}C_1\|u_f(t)\|^2_{L^2(\Omega)} + \frac{1}{2}C_2|f(t)|^2
\end{align*}
where
\begin{equation}\label{c12_constants}
C_1 = \left(\frac{1}{\varepsilon}\|b\|^2_{L^{\infty}([0,T]\times\Omega)} + 2\|c\|_{L^{\infty}([0,T]\times\Omega)} + 1 \right),\;\text{and $C_2 = |\Omega|$}.
\end{equation}
Now, pick $\varepsilon>0$ small enough so that $(\sigma-\varepsilon/2)>0$. By ignoring the positive term on the left-hand side, we conclude that
\begin{align}\label{gronwall}
\frac{d}{dt}\|u_f(t)\|^2_{L^2(\Omega)} \le C_1\|u_f(t)\|^2_{L^2(\Omega)} + C_2|f(t)|^2.
\end{align}
Applying Gronwall's Inequality yields
\begin{equation}\label{uf_bound}
\|u_f(t)\|^2_{L^2(\Omega)} \le C_3\|f\|^2_{L^2(0,T)},
\end{equation}
where $C_3 = e^{C_1T}C_2$ and $C_1$ and $C_2$ are given by~\eqref{c12_constants}.

We are now ready to establish a bound for the forward map $F$. We have for a given source $f$:
\[
\|F(f)\|_{L^2(0,T)}^2 \le |\Omega|\,\int_0^T \|u_f(t)\|^2_{L^2(\Omega)}\,dt\le T|\Omega| C_3\,\|f\|^2_{L^2(0,T)}
\]
where we have used~\eqref{uf_bound}. We conclude that
\begin{equation}\label{F_bound}
\|F(f)\|_{L^2(0,T)} \le C \|f\|_{L^2(0,T)},
\end{equation}
where $C = C\left(\|b\|_{L^{\infty}([0,T]\times\Omega)}, \|c\|_{L^{\infty}([0,T]\times\Omega)}, |\Omega|, T\right)$. This shows that the forward map $F$ is bounded, as needed.
\end{proof}

\begin{remark}
The estimate~\eqref{F_bound} implies, by definition of the operator norm, that
\[
\|F\|_{\mathcal{L}(L^2(0,T))}
:=
\sup_{\substack{f\in L^2(0,T)\\ f\ne0}}
\frac{\|F(f)\|_{L^2(0,T)}}{\|f\|_{L^2(0,T)}}
\le C.
\]
In the analysis of the integral feedback method in Section~\ref{sec:feedback_analysis}, we use the operator-norm notation
$\|F\|_{\mathcal{L}(L^2(0,T))}$ for the bounded forward operator $F$.
\end{remark}

Our next step is to show the existence and uniqueness of solutions. To this end, let
\begin{equation}\label{inverse_data_modified}
h^\delta(t) := \mu^\delta(t) - \int_\Omega u_\zeta(t,x)\,dx - \int_\Omega u_g(t,x)\,dx - \int_\Omega u_{u_0}(t,x)\,dx 
\end{equation}
be the adjusted noisy data after removing known contributions from the source $\zeta$, the initial $u_0$, and the boundary condition $g$. The function $\mu^{\delta}(t)$ is the noisy data from the given data $\mu(t)$. 

We define the Tikhonov functional:
\begin{equation}\label{Tikhonov_functional}
J_\lambda(f) = \|F(f) - h^\delta\|_{L^2(0,T)}^2 + \lambda \|f\|_{L^2(0,T)}^2,
\end{equation}
where $\lambda > 0$ is the regularization parameter. We aim to prove the following result:
\begin{theorem}\label{existence_uniqueness2}
For every $\lambda>0$ and $h^{\delta}\in L^2(0,T)$, the functional $J_\lambda(f)$ admits a unique minimizer in $L^2(0,T)$. That is, there exists a unique $f^\lambda\in L^2(0,T)$ such that
\[
J_\lambda(f^\lambda) = \inf_{f\in L^2(0,T)}J_\lambda(f).
\]
\end{theorem}

Our proof relies on the direct method of the calculus of variations. In a reflexive Banach space $X$, any coercive, weakly lower semicontinuous, and strictly convex functional admits a unique minimizer over $X$ (see, e.g., \cite{EkelandTemam99,Evans10}). Coercivity and weak lower semicontinuity ensure the existence of a minimizer, while strict convexity guarantees uniqueness.

\begin{remark}
From this point on, we treat $h^\delta \in L^2(0,T)$ as a fixed function and suppress its explicit decomposition in terms of $\mu^\delta, \zeta, u_0, g$ for clarity. The Tikhonov functional is analyzed purely in terms of $f$ and $h^\delta$, as the known data components do not affect the mathematical properties such as coercivity, convexity, or weak lower semicontinuity.
\end{remark}

\begin{proof}
Observe that $J_\lambda$ is proper, since $F$ is a bounded linear operator from $L^2(0,T)$ to $L^2(0,T)$ by Lemma~\ref{lem:linearity} and Lemma~\ref{lem:bounded}. In fact, since $h^\delta\in L^2(0,T)$ and $F(f)\in L^2(0,T)$ for all $f\in L^2(0,T)$, both terms in $J_\lambda(f)$ from~\eqref{Tikhonov_functional} are finite for all $f\in L^2(0,T)$. This implies that $J_\lambda(f)<\infty$ for all $f\in L^2(0,T)$, as needed.

To show that $J_\lambda$ is coercive, we have, for all $f\in L^2(0,T)$,
\begin{align}\label{coercivity}
J_\lambda(f) &= \|F(f)\|^2_{L^2(0,T)} - 2\langle F(f), h^\delta\rangle + \|h^\delta\|_{L^2(0,T)}^2 + \lambda \|f\|_{L^2(0,T)}^2 \notag\\
& \ge -2C\|f\|_{L^2(0,T)}\|h^\delta\|_{L^2(0,T)} + \|h^\delta\|_{L^2(0,T)}^2 + \lambda \|f\|_{L^2(0,T)}^2
\end{align}
where in~\eqref{coercivity} we have used the boundedness of $F(f)$ as in~\eqref{F_bound}. Now, observe that the right-hand side of~\eqref{coercivity} is a quadratic function of $\|f\|_{L^2(0,T)}$ and thus $J_\lambda(f)\to\infty$ as $\|f\|_{L^2(0,T)}\to \infty$ since $\lambda>0$, as required.

Next, we will show that the Tikhonov functional $J_\lambda$ is weakly lower semicontinuous (WLSC). To this end, we will show that each of the norms in the definition of $J_\lambda$ is weakly lower semicontinuous in any real or complex Hilbert space $H$. The result will follow in $L^2(0,T)$. 

First, let us show that the map $f\mapsto \|f\|^2_H$ is WLSC. To this end, let $(f_n)$ be a sequence in $H$ and $f\in H$ with $f_n\rightharpoonup f$ in $H$. Recalling the Generalized Polarization Identity
\begin{equation}\label{polarization}
\|f_n\|^2_H = \|f\|^2_H + \|f_n - f\|_H^2 + 2\Re\langle f_n - f, f\rangle,
\end{equation}
we have,
\begin{equation*}
\liminf_{n\to\infty}\|f_n\|^2_H = \liminf_{n\to\infty}\left(\|f\|^2_H + \|f_n - f\|_H^2 + 2\Re\langle f_n - f, f\rangle\right).
\end{equation*}
Since $f_n\rightharpoonup f$ in $H$, then $\langle f_n - f, f\rangle\to\langle 0, f\rangle = 0$ and because $\|f_n - f\|^2_H\ge 0$, it follows that
\begin{equation}\label{norm1_WLSC}
\liminf_{n\to\infty}\|f_n\|^2_H \ge \|f\|^2_H,
\end{equation}
as desired.

Second, we show that if $Y$ is another Hilbert space and $F: H\to Y$ is continuous, then for all $h^\delta\in H$, the map $f\mapsto\|F(f) - h^\delta\|^2_Y$ is WLSC. To this end, taking $H = Y = L^2(0,T)$ we conclude that $F$ is continuous from Lemma~\ref{lem:linearity} and Lemma~\ref{lem:bounded} and hence weakly continuous. Now, since $f_n\rightharpoonup f$ in $H$ then $F(f_n)\rightharpoonup F(f)$ in $Y$. By our previous result,
\begin{equation}\label{norm2_WLSC}
\liminf_{n\to\infty}\|F(f_n) - h^\delta\|^2_Y \ge \|F(f) - h^\delta\|^2_Y,
\end{equation}
as needed. In conclusion, we have, using~\eqref{norm1_WLSC} and ~\eqref{norm2_WLSC}
\begin{align*}
\liminf_{n\to\infty}J_\lambda(f_n)
& \ge \|F(f) - h^\delta\|^2_{L^2(0,T)} + \lambda \|f\|^2_{L^2(0,T)} = J_\lambda(f).
\end{align*}
That is, $J_\lambda(f)$ is WLSC, as needed. 

In summary, we have shown so far that for all $f\in L^2(0,T)$, the Tikhonov functional $J_\lambda$ for $\lambda>0$ is proper, coercive and weakly lower semicontinuous. Thus, there exists $f^\lambda\in L^2(0,T)$ such that
\begin{equation}\label{eq:minimizer}
J_\lambda(f^\lambda) = \min_{f\in L^2(0,T)}J_\lambda(f).
\end{equation}
Finally, we show that the Tikhonov functional $J_\lambda$ is strictly convex. To this end, we will show that the maps $f\mapsto\|f\|^2_{L^2(0,T)}$ and $f\mapsto\|F(f) - h^\delta\|^2_{L^2(0,T)}$ for $h^\delta\in L^2(0,T)$, are strictly convex and convex, respectively.

To show that the map $f\mapsto\|f\|^2_{L^2(0,T)}$ is strictly convex, let $f_1, f_2 \in L^2(0,T)$ with $f_1 \ne f_2$ and $\rho\in (0,1)$. We have
\begin{align*}
\|\rho f_1 + (1-\rho)f_2\|^2_{L^2(0,T)} & = \rho^2\|f_1\|^2 + 2\rho(1-\rho)\langle f_1, f_2\rangle + (1-\rho)^2\|f_2\|^2_{L^2(0,T)}.
\end{align*}
It follows that
\begin{align*}
&\rho\|f_1\|^2_{L^2(0,T)} + (1-\rho)\|f_2\|^2_{L^2(0,T)} - \|\rho f_1 + (1-\rho)f_2\|^2_{L^2(0,T)} \\ 
& = \rho(1-\rho)\left[\|f_1\|^2_{L^2(0,T)} + \|f_2\|^2_{L^2(0,T)} -2\langle f_1, f_2\rangle \right] = \rho(1-\rho)\|f_1 - f_2\|^2_{L^2(0,T)}.
\end{align*}
This implies that
\begin{equation*}
\|\rho f_1 + (1-\rho)f_2\|^2_{L^2(0,T)} = \rho\|f_1\|^2_{L^2(0,T)} + (1-\rho)\|f_2\|^2_{L^2(0,T)} - \rho(1-\rho)\|f_1 - f_2\|^2_{L^2(0,T)}.
\end{equation*}
Since $f_1\ne f_2$ and $\rho\in (0,1)$ then $\rho(1-\rho)\|f_1 - f_2\|^2_{L^2(0,T)}>0$. Thus,
\begin{equation}\label{strict_convex}
\|\rho f_1 + (1-\rho)f_2\|^2_{L^2(0,T)} < \rho\|f_1\|^2_{L^2(0,T)} + (1-\rho)\|f_2\|^2_{L^2(0,T)},
\end{equation}
as desired.

To show the convexity of the map $f\mapsto\|F(f) - h^\delta\|^2_{L^2(0,T)}$, let $f_1, f_2\in L^2(0,T)$ and $\rho\in [0,1]$. Since $F$ is linear, then
\[
F(\rho f_1 + (1-\rho)f_2) - h^\delta = \rho(F(f_1) - h^\delta) + (1-\rho)(F(f_2) - h^\delta).
\]
Employing similar computation as before we obtain,
\begin{align*}
\|F\left(\rho f_1 + (1-\rho)f_2\right) - h^\delta\|^2_{L^2(0,T)} & = \|\rho(F(f_1) - h^\delta) + (1-\rho)(F(f_2) - h^\delta)\|^2_{L^2(0,T)}\\
& = \rho^2\|F(f_1)-h^\delta\|^2_{L^2(0,T)} + (1-\rho)^2\|F(f_2)-h^\delta\|^2_{L^2(0,T)}\\
& + 2\rho(1-\rho)\langle F(f_1)-h^\delta, F(f_2)-h^\delta\rangle. 
\end{align*}
Hence, we get
\begin{align*}
&\rho\|F(f_1)-h^\delta\|^2_{L^2(0,T)} + (1-\rho)\|F(f_2)-h^\delta\|^2_{L^2(0,T)}\\
& - \|\rho(F(f_1) - h^\delta) + (1-\rho)(F(f_2) - h^\delta)\|^2_{L^2(0,T)} = \rho(1-\rho)\|F(f_1) - F(f_2)\|^2_{L^2(0,T)}\ge 0,
\end{align*}
since $\rho\in [0,1]$. This implies that
\begin{align}\label{convexity}
\|\rho(F(f_1) - h^\delta) + (1-\rho)(F(f_2) - h^\delta)\|^2_{L^2(0,T)} &\le \rho\|F(f_1)-h^\delta\|^2_{L^2(0,T)} \notag\\
& + (1-\rho)\|F(f_2)-h^\delta\|^2_{L^2(0,T)},
\end{align}
showing that the map $f\mapsto \|F(f) - h^\delta\|^2_{L^2(0,T)}$ is convex.

We conclude with the strict convexity of the functional $J_\lambda$ for $\lambda>0$. To this end, let $f_1, f_2\in L^2(0,T)$ with $f_1\ne f_2$ and $\rho\in (0,1)$. Then using~\eqref{strict_convex} and~\eqref{convexity} we get the following.
\begin{align}
J_\lambda(\rho f_1 + (1-\rho)f_2) &= \|F\left(\rho f_1 + (1-\rho)f_2\right) - h^\delta\|^2_{L^2(0,T)} + \lambda\|\rho f_1 + (1-\rho)f_2\|^2_{L^2(0,T)}\notag\\
& \le \rho J_\lambda(f_1) + (1-\rho) J_\lambda(f_2).\notag
\end{align}
This shows that the functional $J_\lambda$ for $\lambda>0$ is strictly convex. Hence, the minimizer $f^\lambda$ in~\eqref{eq:minimizer} of $J_\lambda$ is unique. The proof of the theorem is complete.
\end{proof}

\section{Optimality Condition of the Minimizer}\label{sec:optimality}
In this section, we provide the Euler-Lagrange characterization of the unique minimizer $f^\lambda$. This optimality condition provides a variational characterization of the minimizer and reveals the role of the adjoint state in the structure of the solution. In addition, it is instrumental in establishing further regularity properties of the minimizer and provides a theoretical connection to classical adjoint-based reconstruction methods.

Although this characterization is not used directly in the proposed adjoint-free reconstruction method, it serves as an important theoretical benchmark and highlights the differences between adjoint-based and feedback-based approaches.

\begin{theorem}[First-order necessary optimality condition]
Let $\lambda>0$ and $h^\delta\in L^2(0,T)$. Let $f^\lambda\in L^2(0,T)$ be the unique minimizer of the Tikhonov functional defined in~\eqref{Tikhonov_functional} as given by Theorem~\ref{existence_uniqueness2}.
Then $f^\lambda$ satisfies the variational optimality condition
\[
J'(f^\lambda)(\vartheta)=0\qquad \text{for all }\vartheta\in L^2(0,T).
\]
Moreover, there exists an adjoint state $p$ such that the optimality condition can be written as
\[
\lambda f^\lambda(t)+\int_\Omega p(t,x)\,dx=0\quad \text{for a.e. }t\in(0,T).
\]
\end{theorem}

\begin{proof}
Let $\varepsilon>0$ and 
\begin{equation}\label{perturbation_source}
f^{\varepsilon} = f^\lambda + \varepsilon l
\end{equation}
be a perturbation of the minimizer $f^\lambda$ where $l\in L^2(0,T)$ is arbitrary. Let 
$u^\varepsilon = u_{f^\varepsilon}$ denote the corresponding weak solution of \eqref{parab2}-\eqref{bc2} with source term $f^\varepsilon$.
By linearity of the PDE in the source term, we have,
\begin{equation*}
u_{f^\varepsilon} = u_{f^\lambda} + \varepsilon u_l,
\end{equation*}
where $u_l$ is the weak solution of the same PDE with source term the perturbation $l$. Hence,
\begin{equation*}
F(f^\lambda + \varepsilon l)(t) = \int_\Omega u_{f^\lambda + \varepsilon l}(t,x)\, dx = \int_\Omega u_{f^\lambda}(t,x)\, dx + \varepsilon \int_\Omega u_l(t,x)\, dx = F(f^\lambda)(t) + \varepsilon \int_\Omega u_l(t,x)\, dx.
\end{equation*}
This shows that $F$ is Fréchet differentiable, with derivative $DF(f): L^2(0,T) \to L^2(0,T)$ given by:
\begin{equation*}
DF(f)(l)(t) = \int_\Omega u_l(t,x)\, dx.
\end{equation*}
Since $f^\lambda$ is the minimizer, we must have that
\begin{equation}\label{derivative_functional}
\frac{d}{d\varepsilon}\Bigg|_{\varepsilon = 0}J_\lambda(f^\varepsilon) = \frac{d}{d\varepsilon}\Bigg|_{\varepsilon = 0}\left[\int_0^T\left(\int_\Omega u^\varepsilon(t,x)\,dx - h^\delta(t)\right)^2\,dt + \lambda\int_0^T|f^\varepsilon(t)|^2\,dt \right] = 0.
\end{equation}
For the regularized term, we obtain 
\[
\frac{d}{d\varepsilon}\Bigg|_{\varepsilon = 0}\lambda\int_0^T|f^\varepsilon(t)|^2\,dt = \frac{d}{d\varepsilon}\Bigg|_{\varepsilon = 0}\lambda\int_0^T|f^\lambda(t) + \varepsilon l(t)|^2\,dt = 2\lambda\int_0^Tf^\lambda(t)l(t)\,dt.
\]
For the misfit term, we have that
\[
\frac{d}{d\varepsilon}\Bigg|_{\varepsilon = 0}\int_0^T\left(\int_\Omega u^\varepsilon(t,x)\,dx - h^\delta(t)\right)^2\,dt = \int_0^T2\left(\int_\Omega u_{f^\lambda}(t,x)\,dx-h^\delta(t)\right)\left(\int_\Omega v(t,x)\,dx\right)\,dt,
\]
where
\[
v := \frac{d}{d\varepsilon}u^\varepsilon\Bigg|_{\varepsilon = 0} = u_l.
\]
Since $u^\varepsilon$ satisfies PDE~\eqref{parab2}-\eqref{bc2} with source term $f^\varepsilon$ as in~\eqref{perturbation_source}, then $v$ satisfies
\begin{alignat*}{2}
\partial_t v - \nabla \cdot (a(t,x) \nabla v) + b(t,x) \cdot \nabla v + c(t,x) v &= l(t), &&\quad (t,x) \in (0,T)\times \Omega \\
v(0,x) &= 0, &&\quad x \in\Omega  \\
a(t,x) \nabla v \cdot \nu &= 0, &&\quad (t,x) \in (0,T)\times\partial\Omega .
\end{alignat*}
Observe that the derivative of the misfit term is expressed in terms of the variation $v$ which is implicitly defined via the above PDE with unknown source $l(t)$. We therefore need to eliminate it. To this end, let $p(t,x)$ be the solution to the adjoint PDE
\begin{alignat*}{2}
-\partial_t p - \nabla \cdot (a(t,x) \nabla p) - b(t,x) \cdot \nabla p + c(t,x) p &= \int_\Omega u_{f^\lambda}(t,x)\,dx - h^\delta(t), &&\quad (t,x) \in (0,T)\times \Omega \\
p(T,x) &= 0, &&\quad x \in\Omega \\
a(t,x) \nabla p \cdot \nu &= 0, &&\quad (t,x) \in (0,T)\times\partial\Omega .
\end{alignat*}
Note that the quantity to the right-hand side of the above equation
\[
F(f^\lambda)(t)-h^\delta(t)
= \int_\Omega u_{f^\lambda}(t,x)\,dx - h^\delta(t)
\]
is a scalar-valued function of time. When it appears as a source term in the adjoint equation, we identify it with the spatially constant function on $\Omega$, i.e.,
\[
(F(f^\lambda)(t)-h^\delta(t))\,1_\Omega(x).
\]
By the adjoint identity, we have
\[
\int_0^T\left(\int_\Omega u_{f^\lambda}(t,x)\,dx-h^\delta(t)\right)\left(\int_\Omega v(t,x)\,dx\right)\,dt = \int_0^Tl(t)\left(\int_\Omega p(t,x)\,dx\right)\,dt.
\]
Thus, from~\eqref{derivative_functional} it follows that
\[
\int_0^Tl(t)\left(\int_\Omega p(t,x)\,dx + \lambda f^\lambda(t)\right)\,dt = 0\quad\text{for all $l\in L^2(0,T)$}.
\]
Hence,
\[
\int_\Omega p(t,x)\,dx + \lambda f^\lambda(t) = 0 \quad\text{for a.e. $t\in (0,T)$.}
\]
\end{proof}

The following result shows that additional regularity on the adjoint $p$ implies an improved regularity of the unique minimizer. 
\begin{corollary}[Improved regularity of the minimizer]
Assume that the adjoint state satisfies
\[
p\in H^1(0,T;L^2(\Omega)).
\]
Then the minimizer $f^\lambda$ belongs to $H^1(0,T)$.
\end{corollary}
\begin{proof}
Define the mapping
\[
\mathcal{I}:L^2(\Omega)\to\mathbb{R},\qquad 
\mathcal{I}(v)=\int_\Omega v(x)\,dx.
\]
This is a bounded linear functional on $L^2(\Omega)$ since $|\mathcal{I}(v)|\le |\Omega|^{1/2}\|v\|_{L^2(\Omega)}$.

If $p\in H^1(0,T;L^2(\Omega))$, then by the standard Sobolev theory for Banach-valued functions, the composition
$t\mapsto \mathcal{I}(p(t))$ belongs to $H^1(0,T)$ and satisfies
\[
\frac{d}{dt}\mathcal{I}(p(t))
=
\mathcal{I}(\partial_t p(t))
\quad\text{for a.e. }t\in(0,T).
\]
It follows that
\[
\frac{d}{dt} f^\lambda(t) =-\frac1\lambda\int_\Omega \partial_t p(x,t)\,dx
\quad \text{a.e. }t\in(0,T),
\]
and
\[
\Big\|\frac{d}{dt} f^\lambda(t)\Big\|_{L^2(0,T)} \le \frac{|\Omega|^{1/2}}{\lambda}\|\partial_tp\|_{L^2(0,T;L^2(\Omega))}.
\]
Therefore $f^\lambda\in H^1(0,T)$.
\end{proof}

\section{Weak-norm stability and conditional H\"older stability}\label{sec:cond_stab}
In this section, we analyze the stability properties of the inverse source problem. For $f\in L^2(0,T)$, let $u$ be the weak solution of~\eqref{parab2}--\eqref{bc2}, which under our assumptions, satisfies the standard Lions--Magenes regularity
$u\in L^2(0,T;H^1(\Omega))$ and $\partial_t u\in L^2(0,T;(H^1(\Omega))^*)$. Now, recall the observation operator
\begin{equation}\label{eq:obs_F}
(Ff)(t):=\int_\Omega u(t,x)\,dx,\qquad t\in(0,T).
\end{equation}
Consider the reduced problem (difference of two states) with homogeneous data:
\[
\partial_t u + A(t)u = f(t)\,1_\Omega \quad \text{in }(0,T)\times\Omega,
\qquad u(0)=0,\qquad a\nabla u\cdot\nu=0,
\]
where
\begin{equation}\label{eq:A_operator}
A(t)u := -\nabla\cdot(a(t,\cdot)\nabla u) + b(t,\cdot)\cdot\nabla u + c(t,\cdot)u,
\end{equation}
understood in the variational sense with Neumann boundary condition.

Parabolic evolution theory yields an evolution family $U(t,s)$ such that the forced solution admits
Duhamel's formula
\[
u(t)=\int_0^t U(t,s)\big(f(s)\,1_\Omega\big)\,ds.
\]
The existence and properties of the evolution family $U(t,s)$ associated with non-autonomous
parabolic operators, including boundedness on $L^2(\Omega)$, are classical
and can be found, for example, in \cite[Chapter~5]{Pazy1983}.

Applying the observation operator~\eqref{eq:obs_F} and using Fubini's theorem (Fubini’s theorem is justified since $U(t,s)$ is uniformly bounded on $L^2(\Omega)$
and $f\in L^2(0,T)$) give the data operator
\begin{equation}\label{eq:operator_data}
    (Ff)(t)=\int_0^t K(t,s)\,f(s)\,ds,
\end{equation}
where $K(t,s)$ is the Volterra kernel associated with the forward operator and is given by
\begin{equation}\label{eq:kernel}
K(t,s):=\int_\Omega \big(U(t,s)1_\Omega\big)(x)\,dx, \quad 0\le s\le t\le T.
\end{equation}
Moreover, since $U(t,t)=I$, it immediately follows that $K(t,t)=|\Omega|$ for all $t\in[0,T]$.

The following lemma provides an upper bound on $\|U(t,s)\|_{L^2\to L^2}$ and shows that $s\mapsto K(t,s)$ is absolutely continuous. Absolute continuity of $K(t,\cdot)$ is used to ensure that
$\partial_s K(t,\cdot)$ exists as an $L^1(0,t)$ weak derivative. This regularity is precisely what
justifies integration by parts in the Volterra variable $s$.

\begin{lemma}\label{lem:U_and_k}
Assume that Assumptions {\it (A1)} and {\it (A2)} hold. Let $U(t,s)$ denote the evolution family associated with the
homogeneous parabolic operator $A(t)$ in~\eqref{eq:A_operator} under homogeneous Neumann boundary conditions, and let the kernel $K$ be defined by~\eqref{eq:kernel}. Then the following statements hold:
\begin{description}
\item[(a)] There exists a constant $C_U = C_U(\sigma,\|b\|_{L^\infty},\|c\|_{L^\infty},T)>0$ such that
\[
\|U(t,s)\|_{L^2(\Omega)\to L^2(\Omega)} \le C_U,
\quad \text{for all $(t,s)\in \;\Delta:=\{(t,s): 0\le s\le t\le T\}$}.
\]

\item[(b)] For each fixed $t\in(0,T]$, the mapping $s\mapsto K(t,s)$ belongs to $W^{1,1}(0,t)$ and
admits an absolutely continuous representative. Moreover, its weak derivative satisfies
\[
\partial_s K(t,s) = -\int_\Omega \big(U(t,s)c(s,\cdot)\big)(x)\,dx
\quad \text{for a.e. } s\in(0,t),
\]
and there exists a constant $C_K = C_K(\sigma,\|b\|_{L^\infty},\|c\|_{L^\infty},T,|\Omega|)>0$ such that
\[
|\partial_s K(t,s)| \le C_K
\qquad \text{for a.e. } (t,s)\in\Delta.
\]
\end{description}
\end{lemma}

\begin{proof}
    Let $v\in L^2(\Omega)$ and consider $u(t,\cdot)=U(t,s)v$ the unique weak solution of
\[
\partial_t u(t,\cdot) + A(t)u(t,\cdot)=0
\quad\text{in }(s,T)\times\Omega,
\qquad a(t,\cdot)\nabla u(t,\cdot)\cdot\nu=0,
\qquad u(s,\cdot)=v.
\]
By mimicking the computation carried out in~\eqref{weak_form}--\eqref{gronwall} with obvious modifications we get
\[
\frac{d}{dt}\|u(t,\cdot)\|_{L^2(\Omega)}^2 \le \gamma \|u(t,\cdot)\|_{L^2(\Omega)}^2,
\qquad
\gamma:=2\|c\|_{L^\infty((0,T)\times\Omega)}+\frac{\|b\|_{L^\infty((0,T)\times\Omega)}^2}{\sigma}.
\]
By Gr\"onwall,
\[
\|U(t,s)v\|_{L^2(\Omega)}\le e^{\frac{\gamma}{2}(t-s)}\|v\|_{L^2(\Omega)}.
\]
Therefore, taking the supremum yields
\[
\|U(t,s)\|_{L^2\to L^2}\le e^{\frac{\gamma}{2}(t-s)}\le e^{\frac{\gamma}{2}T}=:C_U.
\]
This establishes part(a). To show part(b), observe that since $\partial_sU(t,s) = -U(t,s)A(s)$ in the operator sense on $L^2(\Omega)$, then~\eqref{eq:kernel} implies that $\partial_sK(t,s)$ exists and
\[
\partial_sK(t,s) = -\int_\Omega \left(U(t,s)c(s,\cdot)\right)(x)\,dx.
\]
Using Cauchy-Schwarz and the result of part(a), we obtain
\[
|\partial_sK(t,s)| \le |\Omega|^{1/2}\|U(t,s)c(s,\cdot)\|_{L^2(\Omega)} \le |\Omega|^{1/2}C_U\|c(s,\cdot)\|_{L^2(\Omega)} \le |\Omega|C_U\|c\|_{L^\infty((0,T)\times\Omega)}=:C_K.
\]
So, for each fixed $t$, $|\partial_sK(t,s)| \le C_K$ for a.e. $s\in (0,t)$, where $C_K$ is a constant independent of $s$. Thus, $\partial_sK(t,s) \in L^\infty(0,t)$, which implies that $\partial_sK(t,s) \in L^1(0,t)$. Hence, $K(t,\cdot)\in W^{1,1}(0,t)$ and has an absolutely continuous representative
\[
K(t,s) = K(t,0) + \int_0^s\partial_\chi K(t,\chi)\,d\chi.
\]
This shows part(b) and completes the proof of the lemma.
\end{proof}

The observation operator $F$ as defined in~\eqref{eq:operator_data} is a Volterra operator in time and therefore acts as a time-integration
operator on the unknown source. As a consequence, high-frequency oscillations in the source are
strongly damped in the data, leading to a loss of information. This smoothing effect prevents any
estimate of Lipschitz stability in $L^2(0,T)$ and is the hallmark of severe ill-posedness. Only weaker
norms, such as $H^{-1}(0,T)$, can be controlled by the data without additional assumptions. This is the purpose of the following lemma.

\begin{lemma}[Weak-norm stability estimate]\label{lem:LB_Hminus1}
Assume that the forward operator $F:L^2(0,T)\to L^2(0,T)$ admits a Volterra representation
\[
(Ff)(t)=\int_0^t K(t,s)f(s)\,ds,
\]
where the kernel $K$ satisfies the properties stated in Lemma~\ref{lem:U_and_k}.
Then there exists a constant $C_0>0$ such that
\[
\|f\|_{H^{-1}(0,T)} \le C_0\,\|F(f)\|_{L^2(0,T)}
\qquad \text{for all } f\in L^2(0,T).
\]
\end{lemma}

\begin{proof}
Let $f\in L^2(0,T)$ and define $Q(t):=\int_0^t f(s)\,ds$, so that $Q\in H^1(0,T)$, $Q(0)=0$, and $Q'=f$.
By definition of the dual norm and integration by parts, using $\varphi(0)=\varphi(T)=0$ for $\varphi\in H_0^1(0,T)$), we have
\[
\left|\int_0^T f\varphi\right|
\le \|Q\|_{L^2(0,T)}\|\varphi'\|_{L^2(0,T)}
\le \|Q\|_{L^2(0,T)}\|\varphi\|_{H_0^1(0,T)}.
\]
Taking the supremum yields
\begin{equation}\label{inequality1}
\|f\|_{H^{-1}(0,T)}\le \|Q\|_{L^2(0,T)}.
\end{equation}
Next, we find an estimate for $\|Q\|_{L^2(0,T)}$. Employing~\eqref{eq:operator_data} we obtain
\[
(Ff)(t)=\int_0^t K(t,s)\,Q'(s)\,ds.
\]
Using the fact that $K(t,\cdot)\in W^{1,1}(0,t)$ from Lemma~\ref{lem:U_and_k}(b), an integration by parts in $s$ yields a Volterra equation of the second kind for $Q$:
\begin{equation}\label{eq:Q_eq}
Q(t)=\frac{1}{|\Omega|}(Ff)(t)+\frac{1}{|\Omega|}\int_0^t \partial_sK(t,s)Q(s)\,ds.
\end{equation}
Let $z(t):=\frac{1}{|\Omega|}(Ff)(t)$. Using Lemma~\ref{lem:U_and_k}(b) and $\beta:= C_K/|\Omega|$, Equation~\eqref{eq:Q_eq}
implies 
\[
|Q(t)|\le |z(t)|+\beta\int_0^t |Q(s)|\,ds.
\]
Define $m(t):=\int_0^t |Q(s)|\,ds$. Then $m'(t)=|Q(t)|$ a.e. and
\[
m'(t)\le |z(t)|+\beta m(t),\qquad m(0)=0.
\]
By Gr\"onwall's inequality,
\[
m(t)\le \int_0^t e^{\beta(t-s)}|z(s)|\,ds.
\]
Therefore,
\[
|Q(t)|\le |z(t)|+\beta\int_0^t e^{\beta(t-s)}|z(s)|\,ds = |z(t)| + (w*|z|)(t), \quad w(t):=\beta e^{\beta t}\mathbf{1}_{t\ge 0}.
\]
Taking $L^2(0,T)$-norms and using the triangle inequality, then applying Young's inequality for convolutions~\cite[Chapter 1]{Grafakos2014} ($L^1*L^2\to L^2$) yields
\[
\|Q\|_{L^2(0,T)}\le \|z\|_{L^2(0,T)} + (e^{\beta T}-1)\|z\|_{L^2(0,T)}
= e^{\beta T}\|z\|_{L^2(0,T)}.
\]
Thus,
\[
\|Q\|_{L^2(0,T)} \le \frac{e^{\beta T}}{|\Omega|}\|F(f)\|_{L^2(0,T)}.
\]
Combining this inequality with~\eqref{inequality1} proves
the lemma with $C_0:=\frac{e^{\beta T}}{|\Omega|}$.
\end{proof}

Stability in stronger norms cannot hold unconditionally for this inverse source problem, due to the Volterra-type smoothing of the forward map.The following theorem
imposes a priori smoothness to establish conditional stability in $L^2(0,T)$.
\begin{theorem}[Conditional H\"older stability on $H^r$-bounded sets]\label{thm:conditional_stability}
Assume that \textit{(A1)--(A3)} hold, and let $F:L^2(0,T)\to L^2(0,T)$ denote the forward operator
defined by
\[
(Ff)(t) := \int_\Omega u_f(x,t)\,dx,
\]
where $u_f$ is the unique weak solution of the parabolic problem
\eqref{parab1}--\eqref{bc} corresponding to the source $f$. Let $f_1,f_2\in L^2(0,T)$ be two admissible sources satisfying the a priori bounds
\begin{equation}\label{eq:source_bound}
\|f_1\|_{H^r(0,T)}\le M,\qquad \|f_2\|_{H^r(0,T)}\le M,
\end{equation}
for some $r>0$ and $M>0$. Then there exists a constant $\widetilde{C}>0$ depending only on $M$, $r$, $T$, the
coefficients of the operator, and $\Omega$, such that
\begin{equation}\label{eq:cond_holder_final}
\|f_1-f_2\|_{L^2(0,T)}
\le
\widetilde{C}\,\|F(f_1)-F(f_2)\|_{L^2(0,T)}^{\frac{r}{r+1}}.
\end{equation}
\end{theorem}

\begin{proof}
According to~\cite[Theorem 6.4.5 (4) p.153]{BerghLofstrom1976}, the real interpolation between the Bessel potential spaces $H^{s_0}_p$ and $H^{s_1}_p$ is given by 
\[
(H^{s_0}_p, H^{s_1}_p)_{\theta, p} = B^{s^*}_{p, q}, \quad \text{where } s^* = (1-\theta)s_0 + \theta s_1,\quad 0<\theta<1.
\]
For the Hilbertian case $p=2$, we have for all $s \in \mathbb{R}$: $H^s_2 = H^s$ (standard Sobolev space) and $B^s_{2,2} = H^s$ (coincidence of Besov and Sobolev spaces at $p=q=2$).
By setting $s_0 = -1$ and $s_1 = r$, and choosing $\theta = \frac{1}{r+1}$ such that $s^* = 0$, we obtain
\[
(H^{-1}, H^r)_{\theta, 2} = B^0_{2,2} = L^2.
\]
Moreover, from~\cite[Eq.~(6) Chapter 2 p.27]{BerghLofstrom1976} we have the interpolation property for operators: if a linear operator $T$ maps $A_0\to B_0$ and $A_1\to B_1$, then it maps the interpolated space $A_\theta\to B_\theta$ with the norm inequality
\begin{equation}\label{eq:Interpo_ineq_operator}
\|T\|_{A_\theta \to B_\theta} \leq C_{\mathrm{int}}\, \|T\|_{A_0 \to B_0}^{1-\theta} \|T\|_{A_1 \to B_1}^\theta,
\end{equation}
for some constant $C_{\mathrm{int}}>0$. Now, set the source spaces $A_0 = A_1 = A_\theta = \mathbb{C}$ (the field of complex numbers) and 
the target spaces as $B_0 = H^{-1}$ and $B_1 = H^r$. Also, set the interpolation space $B_\theta = (H^{-1}, H^r)_{\theta, 2}$ where $\theta$ is chosen as above.

Next, for a fixed function $u \in X$, where $X$ is any Banach space, we define the linear operator $T: \mathbb{C} \to X$ by $T(c) = c \cdot u$ for $c \in \mathbb{C}$.
It follows that $\|T\|_{\mathbb{C} \to X} = \|u\|_X$. Plugging these results into the general property~\eqref{eq:Interpo_ineq_operator}, we obtain
\[
\|u\|_{L^2(0,T)}\le C_{\mathrm{int}}\,\|u\|_{H^{-1}(0,T)}^{\frac{r}{r+1}}\,\|u\|_{H^{r}(0,T)}^{\frac{1}{r+1}}.
\]
Applying this inequality to $f_1 - f_2$ yields
\[
\|f_1 - f_2\|_{L^2(0,T)}\le C_{\mathrm{int}}\,\|f_1 - f_2\|_{H^{-1}(0,T)}^{\frac{r}{r+1}}\,
\|f_1 - f_2\|_{H^r(0,T)}^{\frac{1}{r+1}}.
\]
Using the triangle inequality with~\eqref{eq:source_bound} and applying Lemma~\ref{lem:LB_Hminus1}, we obtain
\[
\|f_1 - f_2\|_{L^2(0,T)}\le C_{\mathrm{int}}\,
\bigl(C_0\|F(f_1) - F(f_2)\|_{L^2(0,T)}\bigr)^{\frac{r}{r+1}}\,
(2M)^{\frac{1}{r+1}}.
\]
where we have also used the linearity of $F$. This yields \eqref{eq:cond_holder_final} with
$\widetilde{C}:=C_{\mathrm{int}}\,C_0^{\frac{r}{r+1}}(2M)^{\frac{1}{r+1}}$.
\end{proof}

\section{Mathematical analysis of the integral feedback method}
\label{sec:feedback_analysis}

Motivated by the Volterra structure of the forward operator established in
Section~\ref{sec:cond_stab}, we now analyze the proposed adjoint-free integral
feedback reconstruction method. Rather than viewing the method merely as an
iterative numerical procedure, we formulate it as a nonlinear fixed-point
iteration on the unknown source. This viewpoint provides a natural mathematical
framework for studying the well-definedness, fixed-point properties, stability,
and convergence of the reconstruction method, while establishing a direct
connection with the analysis developed in the previous sections.

Recall that the inverse problem has been reduced to $F(f)=h^\delta,$ where $F:L^2(0,T)\to L^2(0,T)$ is the adjusted forward operator
\[
F(f)(t)=\int_\Omega u_f(t,x)\,dx,
\]
and
\[
h^\delta(t)
=
\mu^\delta(t)
-
\int_\Omega u_\zeta(t,x)\,dx
-
\int_\Omega u_g(t,x)\,dx
-
\int_\Omega u_{u_0}(t,x)\,dx
\]
denotes the adjusted noisy data obtained after removing the known contributions
of the distributed source $\zeta$, the boundary data $g$, and the initial
condition $u_0$ from the noisy observation $\mu^\delta$.

Let $f^{(k)}\in L^2(0,T)$ denote the current approximation of the unknown source.
We define the corresponding residual by
\[
r^{(k)}(t)
:=
F(f^{(k)})(t)-h^\delta(t),
\qquad t\in(0,T).
\]
The integral feedback method updates the source pointwise according to
\begin{equation}\label{eq:feedback_update_pointwise}
f^{(k+1)}(t)
=
f^{(k)}(t)
-
\frac{1}{\alpha}\,
\phi\bigl(r^{(k)}(t)\bigr),
\qquad t\in(0,T),
\end{equation}
where $\alpha>0$ is the feedback parameter and
$\phi:\mathbb{R}\to\mathbb{R}$ is the feedback function.

In classical PDE-constrained optimization methods, source reconstruction typically requires the computation of the adjoint operator $F^*$, which entails solving an additional adjoint PDE backward in time. In contrast, the proposed method avoids adjoint equations entirely and updates the source directly from the data misfit, relying only on forward solves. This adjoint-free formulation eliminates the additional adjoint solve and leads to a comparatively simple reconstruction procedure.

To formulate~\eqref{eq:feedback_update_pointwise} in operator form, let
$\Phi:L^2(0,T)\to L^2(0,T)$ denote the Nemytskii (or superposition) operator
induced pointwise by $\phi$ (see, e.g.,~\cite{AppellZabrejko1990}), namely
\[
(\Phi v)(t):=\phi(v(t)).
\]
We then define the integral feedback operator
\[
\mathcal G:L^2(0,T)\to L^2(0,T)
\]
by
\[
\mathcal G(f)
:=
f-\frac{1}{\alpha}\Phi\bigl(F(f)-h^\delta\bigr).
\]
Thus, starting from an initial guess $f^{(0)}\in L^2(0,T)$, the feedback
iteration can be written equivalently as
\[
f^{(k+1)}
=
\mathcal G(f^{(k)}),
\qquad k=0,1,2,\ldots.
\]

We assume throughout the analysis that the feedback function
$\phi:\mathbb{R}\to\mathbb{R}$ is globally Lipschitz; that is, there exists
$L_\phi>0$ such that
\[
|\phi(s)-\phi(r)|
\le
L_\phi|s-r|,
\qquad s,r\in\mathbb{R}.
\]
The global Lipschitz continuity of $\phi$ implies that
\[
\|\Phi(v_1)-\Phi(v_2)\|_{L^2(0,T)}
\le
L_\phi
\|v_1-v_2\|_{L^2(0,T)}
\]
for all $v_1,v_2\in L^2(0,T)$.

\begin{theorem}[Well-definedness and Lipschitz continuity of the feedback operator]
\label{thm:feedback_well_defined}
Let $h^\delta\in L^2(0,T)$ and $\alpha>0$, and suppose that
$\phi:\mathbb{R}\to\mathbb{R}$ is globally Lipschitz with constant
$L_\phi$. Let $F:L^2(0,T)\to L^2(0,T)$ be the bounded linear forward
operator defined in~\eqref{forward_map2}. Define
\[
\mathcal{G}:L^2(0,T)\to L^2(0,T)
\]
by
\begin{equation}\label{eq:feedback_operator}
\mathcal{G}(f)
:=
f-\frac{1}{\alpha}
\Phi\bigl(F(f)-h^\delta\bigr).
\end{equation}
Then $\mathcal{G}$ is well defined. Moreover, for every
$f_1,f_2\in L^2(0,T)$,
\begin{equation}\label{eq:G_Lipschitz}
\|\mathcal{G}(f_1)-\mathcal{G}(f_2)\|_{L^2(0,T)}
\le
\left(
1+\frac{L_\phi\|F\|_{\mathcal{L}(L^2(0,T))}}{\alpha}
\right)
\|f_1-f_2\|_{L^2(0,T)}.
\end{equation}
In particular, $\mathcal{G}$ is globally Lipschitz continuous on
$L^2(0,T)$.
\end{theorem}

\begin{proof}
Let $f\in L^2(0,T)$. Since the forward operator
$F:L^2(0,T)\to L^2(0,T)$ is bounded, we have $F(f)\in L^2(0,T).$
Since $h^\delta\in L^2(0,T)$, it follows that $F(f)-h^\delta\in L^2(0,T).$

We first verify that the Nemytskii operator $\Phi$ maps $L^2(0,T)$
into itself. Since $\phi$ is globally Lipschitz,
\[
|\phi(s)|
\le
|\phi(s)-\phi(0)|+|\phi(0)|
\le
L_\phi |s|+|\phi(0)|
\qquad\text{for all }s\in\mathbb{R}.
\]
Hence, for every $v\in L^2(0,T)$,
\[
|\Phi(v)(t)|
=
|\phi(v(t))|
\le
L_\phi |v(t)|+|\phi(0)|.
\]
Since $T<\infty$, the constant function $t\mapsto\phi(0)$ belongs to
$L^2(0,T)$. Therefore, $\Phi(v)\in L^2(0,T).$
Taking $v=F(f)-h^\delta$, we obtain $\Phi\bigl(F(f)-h^\delta\bigr)\in L^2(0,T).$ Consequently,
\[
\mathcal{G}(f)
=
f-\frac{1}{\alpha}
\Phi\bigl(F(f)-h^\delta\bigr)
\in L^2(0,T),
\]
which proves that $\mathcal{G}$ is well defined.

We next establish the Lipschitz estimate. Let
$f_1,f_2\in L^2(0,T)$. Using~\eqref{eq:feedback_operator}, the triangle inequality, and the
Lipschitz continuity of $\Phi$, we obtain
\begin{align*}
\|\mathcal{G}(f_1)-\mathcal{G}(f_2)\|_{L^2(0,T)}
&\le
\|f_1-f_2\|_{L^2(0,T)}
+
\frac{L_\phi}{\alpha}
\|F(f_1)-F(f_2)\|_{L^2(0,T)}.
\end{align*}
By linearity and boundedness of $F$,
\[
\|F(f_1)-F(f_2)\|_{L^2(0,T)}
=
\|F(f_1-f_2)\|_{L^2(0,T)}
\le
\|F\|_{\mathcal{L}(L^2(0,T))}
\|f_1-f_2\|_{L^2(0,T)}.
\]
Hence,
\[
\|\mathcal{G}(f_1)-\mathcal{G}(f_2)\|_{L^2(0,T)}
\le
\left(
1+\frac{L_\phi\|F\|_{\mathcal{L}(L^2(0,T))}}{\alpha}
\right)
\|f_1-f_2\|_{L^2(0,T)},
\]
which proves~\eqref{eq:G_Lipschitz} and completes the proof.
\end{proof}

We next characterize the fixed points of the feedback operator. In
addition to the global Lipschitz continuity of $\phi$, we assume that
\begin{equation}\label{eq:feedback_zero}
\phi(s)=0
\quad\Longleftrightarrow\quad
s=0.
\end{equation}

\begin{theorem}[Characterization of fixed points]
\label{thm:feedback_fixed_points}
Let $\mathcal{G}$ be the feedback operator defined
in~\eqref{eq:feedback_operator}, and suppose that
$\phi$ satisfies~\eqref{eq:feedback_zero}. Then
$\tilde{f}\in L^2(0,T)$ is a fixed point of $\mathcal{G}$ if and only if
\[
F(\tilde{f})=h^\delta
\quad\text{a.e. on }(0,T).
\]
\end{theorem}

\begin{proof}
Suppose first that $\tilde{f}$ is a fixed point of $\mathcal{G}$. Then $\mathcal{G}(\tilde{f})=\tilde{f}.$
Using the definition of $\mathcal{G}$ in~\eqref{eq:feedback_operator},
we obtain $\Phi\bigl(F(\tilde{f})-h^\delta\bigr)=0,$
which means that
\[
\phi\bigl(F(\tilde{f})(t)-h^\delta(t)\bigr)=0
\quad\text{for a.e. }t\in(0,T).
\]
By~\eqref{eq:feedback_zero}, it follows that $F(\tilde{f})(t)-h^\delta(t)=0$ for a.e. $t\in(0,T)$ and therefore $F(\tilde{f})=h^\delta.$

Conversely, suppose that $F(\tilde{f})=h^\delta.$ Then $F(\tilde{f})-h^\delta=0.$

Since~\eqref{eq:feedback_zero} implies $\phi(0)=0$, we have
$\Phi\bigl(F(\tilde{f})-h^\delta\bigr)=0$.
Therefore,
\[
\mathcal{G}(\tilde{f})
=
\tilde{f}-\frac{1}{\alpha}
\Phi\bigl(F(\tilde{f})-h^\delta\bigr)
=
\tilde{f}.
\]
Thus $\tilde{f}$ is a fixed point of $\mathcal{G}$.
\end{proof}

\begin{remark}
Theorem~\ref{thm:feedback_fixed_points} shows that the inverse problem
$F(f)=h^\delta$ can equivalently be formulated as a fixed-point
problem for the feedback operator $\mathcal{G}$. For noisy data,
however, $h^\delta$ need not belong to the range of $F$, and an exact
fixed point may therefore fail to exist. This motivates terminating
the feedback iteration according to a discrepancy principle rather
than requiring convergence to an exact solution of the noisy inverse
problem.
\end{remark}

We next examine the evolution of the data residual under the feedback
iteration. This identity will provide the starting point for the
convergence analysis.

\begin{lemma}[Residual evolution]
\label{lem:residual_evolution}
Let $h\in L^2(0,T)$ be exact data satisfying
\[
h=F(f^\dagger)
\]
for some $f^\dagger\in L^2(0,T)$. Let the sequence
$\{f^{(k)}\}_{k\geq 0}$ be generated by
\[
f^{(k+1)}
=
f^{(k)}
-\frac{1}{\alpha}
\Phi\bigl(F(f^{(k)})-h\bigr),
\]
and define the residual
\[
r^{(k)}:=F(f^{(k)})-h.
\]
Then
\begin{equation}\label{eq:residual_recursion}
r^{(k+1)}
=
r^{(k)}
-\frac{1}{\alpha}F\Phi(r^{(k)}).
\end{equation}
Moreover,
\begin{equation}\label{eq:residual_energy}
\|r^{(k+1)}\|_{L^2(0,T)}^2
=
\|r^{(k)}\|_{L^2(0,T)}^2
-\frac{2}{\alpha}
\left\langle
r^{(k)},F\Phi(r^{(k)})
\right\rangle_{L^2(0,T)}
+
\frac{1}{\alpha^2}
\|F\Phi(r^{(k)})\|_{L^2(0,T)}^2.
\end{equation}
\end{lemma}

\begin{proof}
Using the feedback iteration and the linearity of $F$, we obtain
\[
r^{(k+1)}
=
F\left(
f^{(k)}
-\frac{1}{\alpha}\Phi(r^{(k)})
\right)-h
=
F(f^{(k)})-h
-\frac{1}{\alpha}F\Phi(r^{(k)})
=
r^{(k)}
-\frac{1}{\alpha}F\Phi(r^{(k)}),
\]
which proves~\eqref{eq:residual_recursion}. Taking the squared $L^2(0,T)$-norm on both sides of
\eqref{eq:residual_recursion} gives
\begin{align*}
\|r^{(k+1)}\|_{L^2(0,T)}^2
&=
\left\|
r^{(k)}
-\frac{1}{\alpha}F\Phi(r^{(k)})
\right\|_{L^2(0,T)}^2\\
&=
\|r^{(k)}\|_{L^2(0,T)}^2
-\frac{2}{\alpha}
\left\langle
r^{(k)},F\Phi(r^{(k)})
\right\rangle_{L^2(0,T)}
+
\frac{1}{\alpha^2}
\|F\Phi(r^{(k)})\|_{L^2(0,T)}^2.
\end{align*}
This proves~\eqref{eq:residual_energy}.
\end{proof}

Identity~\eqref{eq:residual_energy} shows that boundedness of the
forward operator and Lipschitz continuity of the feedback function
alone do not imply monotonic decay of the residual. Indeed, such a
conclusion requires additional control of the interaction term
\[
\left\langle r^{(k)},F\Phi(r^{(k)})\right\rangle_{L^2(0,T)}.
\]
Unlike the classical Landweber iteration, the proposed feedback method
does not involve the adjoint operator $F^*$. Consequently, the usual
Landweber descent argument is not directly available. In the following,
we exploit instead the Volterra structure of $F$ established in
Section~\ref{sec:cond_stab}.

We next investigate the stability of the feedback iteration with
respect to perturbations in the observed data. Let $h\in L^2(0,T)$
denote the exact adjusted data and let $h^\delta\in L^2(0,T)$ satisfy
\begin{equation}\label{eq:data_noise_feedback}
\|h^\delta-h\|_{L^2(0,T)}\le \delta.
\end{equation}

\begin{theorem}[Finite-iteration stability with respect to noisy data]
\label{thm:finite_iteration_stability}
Let $F:L^2(0,T)\to L^2(0,T)$ be the bounded linear forward operator,
let $\alpha>0$, and suppose that the feedback function $\phi$ is
globally Lipschitz with constant $L_\phi$. Let
$\{f^{(k)}\}_{k\geq0}$ and $\{f_\delta^{(k)}\}_{k\geq0}$ be generated,
respectively, by
\begin{equation}\label{eq:exact_feedback_iteration}
f^{(k+1)}
=
f^{(k)}
-\frac{1}{\alpha}
\Phi\bigl(F(f^{(k)})-h\bigr)
\end{equation}
and
\begin{equation}\label{eq:noisy_feedback_iteration}
f_\delta^{(k+1)}
=
f_\delta^{(k)}
-\frac{1}{\alpha}
\Phi\bigl(F(f_\delta^{(k)})-h^\delta\bigr).
\end{equation}
Assume that both iterations start from the same initial guess, $f_\delta^{(0)}=f^{(0)}.$ Then, for every $k\geq1$,
\begin{equation}\label{eq:finite_iteration_stability}
\|f_\delta^{(k)}-f^{(k)}\|_{L^2(0,T)}
\le
\frac{L_\phi\delta}{\alpha}
\sum_{j=0}^{k-1}
\left(
1+\frac{L_\phi\|F\|_{\mathcal{L}(L^2(0,T))}}{\alpha}
\right)^j.
\end{equation}
Equivalently,
\begin{equation}\label{eq:finite_iteration_stability_explicit}
\|f_\delta^{(k)}-f^{(k)}\|_{L^2(0,T)}
\le
\frac{\delta}{\|F\|_{\mathcal{L}(L^2(0,T))}}
\left[
\left(
1+\frac{L_\phi\|F\|_{\mathcal{L}(L^2(0,T))}}{\alpha}
\right)^k
-1
\right].
\end{equation}
Consequently, for every fixed $k$,
\[
f_\delta^{(k)}\longrightarrow f^{(k)}
\quad\text{in }L^2(0,T)
\qquad\text{as }\delta\to0.
\]
\end{theorem}

\begin{proof}
Define $d^{(k)}:=f_\delta^{(k)}-f^{(k)}.$
Subtracting~\eqref{eq:exact_feedback_iteration} from
\eqref{eq:noisy_feedback_iteration} gives
\begin{align*}
d^{(k+1)}
&=
d^{(k)}
-\frac{1}{\alpha}
\Big[
\Phi\bigl(F(f_\delta^{(k)})-h^\delta\bigr)
-
\Phi\bigl(F(f^{(k)})-h\bigr)
\Big].
\end{align*}
Using the triangle inequality and the Lipschitz continuity of $\Phi$,
we obtain
\[
\|d^{(k+1)}\|_{L^2(0,T)}
\le
\|d^{(k)}\|_{L^2(0,T)}+
\frac{L_\phi}{\alpha}
\|
F(f_\delta^{(k)})-F(f^{(k)})
-(h^\delta-h)
\|_{L^2(0,T)}.
\]
Using the linearity of $F$, ~\eqref{eq:data_noise_feedback}, and the boundedness of $F$, we obtain
\begin{align*}
\|d^{(k+1)}\|_{L^2(0,T)}
&\le
\|d^{(k)}\|_{L^2(0,T)}
+
\frac{L_\phi}{\alpha}
\left(
\|F(d^{(k)})\|_{L^2(0,T)}
+\delta
\right)
\\
&\le
\left(
1+\frac{L_\phi\|F\|_{\mathcal{L}(L^2(0,T))}}{\alpha}
\right)
\|d^{(k)}\|_{L^2(0,T)}
+
\frac{L_\phi}{\alpha}\delta.
\end{align*}
Set
\[
q:=
1+\frac{L_\phi\|F\|_{\mathcal{L}(L^2(0,T))}}{\alpha}.
\]
Then
\begin{equation}\label{eq:noise_recursion}
\|d^{(k+1)}\|_{L^2(0,T)}
\le
q\|d^{(k)}\|_{L^2(0,T)}
+
\frac{L_\phi}{\alpha}\delta.
\end{equation}
Since the two iterations start from the same initial guess,
$d^{(0)}=0$. Repeated application of~\eqref{eq:noise_recursion}
therefore yields
\[
\|d^{(k)}\|_{L^2(0,T)}
\le
\frac{L_\phi\delta}{\alpha}
\sum_{j=0}^{k-1}q^j,
\]
which proves~\eqref{eq:finite_iteration_stability}.
Since
\[
q-1
=
\frac{L_\phi\|F\|_{\mathcal{L}(L^2(0,T))}}{\alpha},
\]
the finite geometric sum gives
\[
\frac{L_\phi\delta}{\alpha}
\sum_{j=0}^{k-1}q^j
=
\frac{\delta}{\|F\|_{\mathcal{L}(L^2(0,T))}}
(q^k-1),
\]
which proves~\eqref{eq:finite_iteration_stability_explicit}.

Finally, for every fixed $k$, the factor multiplying $\delta$ is
independent of $\delta$. Hence,
\[
\|f_\delta^{(k)}-f^{(k)}\|_{L^2(0,T)}
\to0
\qquad\text{as }\delta\to0.
\]
This completes the proof.
\end{proof}

Theorem~\ref{thm:finite_iteration_stability} shows that every fixed
number of feedback iterations is stable with respect to perturbations
in the data. At the same time, the upper bound in
\eqref{eq:finite_iteration_stability_explicit} grows with the iteration
index $k$, indicating the possibility of increasing noise propagation
under prolonged iteration. This provides a theoretical motivation for
terminating the reconstruction before excessive fitting of the noisy
data occurs. In the numerical implementation, this is accomplished
using the discrepancy principle.

We now establish an exact-data convergence result under an
accretivity condition on the feedback direction. This condition
ensures that the correction generated by the feedback operator is
sufficiently aligned with the current data residual.

\begin{theorem}[Exact-data convergence of the feedback iteration]
\label{thm:exact_feedback_convergence}
Let $h=F(f^\dagger)$ for some $f^\dagger\in L^2(0,T)$, and let
$\{f^{(k)}\}_{k\geq0}$ be generated by
\[
f^{(k+1)}
=
f^{(k)}
-\frac{1}{\alpha}
\Phi\bigl(F(f^{(k)})-h\bigr).
\]
Define
\[
r^{(k)}:=F(f^{(k)})-h.
\]
Assume that $\phi(0)=0$ and that $\phi$ is globally Lipschitz with
constant $L_\phi$. Suppose further that there exists $\widetilde{\gamma}>0$ such that
\begin{equation}\label{eq:feedback_accretivity}
\left\langle
r^{(k)},F\Phi(r^{(k)})
\right\rangle_{L^2(0,T)}
\ge
\widetilde{\gamma}\|r^{(k)}\|_{L^2(0,T)}^2
\end{equation}
for every $k\geq0$. If
\begin{equation}\label{eq:alpha_convergence_condition}
\alpha>
\frac{L_\phi^2
\|F\|_{\mathcal{L}(L^2(0,T))}^2}{2\widetilde{\gamma}},
\end{equation}
then there exists $\widetilde{q}\in[0,1)$ such that
\begin{equation}\label{eq:residual_geometric_decay}
\|r^{(k)}\|_{L^2(0,T)}
\le
\widetilde{q}^k\|r^{(0)}\|_{L^2(0,T)},
\end{equation}
where
\[
\widetilde{q}^2
=
1-\frac{2\widetilde{\gamma}}{\alpha}
+
\frac{L_\phi^2
\|F\|_{\mathcal{L}(L^2(0,T))}^2}{\alpha^2}.
\]

Assume, in addition, that for some $r>0$ and $M>0$,
\[
\|f^\dagger\|_{H^r(0,T)}\le M,
\qquad
\|f^{(k)}\|_{H^r(0,T)}\le M
\quad\text{for all }k\geq0.
\]
Then
\begin{equation}\label{eq:source_feedback_convergence}
\|f^{(k)}-f^\dagger\|_{L^2(0,T)}
\le
\widetilde{C}
\widetilde{q}^{k\frac{r}{r+1}}
\|r^{(0)}\|_{L^2(0,T)}^{\frac{r}{r+1}},
\end{equation}
where $\widetilde{C}>0$ is the constant in the conditional H\"older stability
estimate of Theorem~\ref{thm:conditional_stability}. In particular,
\[
f^{(k)}\to f^\dagger
\qquad\text{in }L^2(0,T)
\quad\text{as }k\to\infty.
\]
\end{theorem}

\begin{proof}
By Lemma~\ref{lem:residual_evolution},
\[
\|r^{(k+1)}\|_{L^2(0,T)}^2
=
\|r^{(k)}\|_{L^2(0,T)}^2
-\frac{2}{\alpha}
\left\langle
r^{(k)},F\Phi(r^{(k)})
\right\rangle_{L^2(0,T)}
+
\frac{1}{\alpha^2}
\|F\Phi(r^{(k)})\|_{L^2(0,T)}^2.
\]
Using~\eqref{eq:feedback_accretivity}, $\phi(0)=0$, $\phi$ is globally Lipschitz, and the boundedness of $F$, we obtain
\[
\|r^{(k+1)}\|_{L^2(0,T)}^2
\le
\left[
1-\frac{2\widetilde{\gamma}}{\alpha}
+
\frac{L_\phi^2
\|F\|_{\mathcal{L}(L^2(0,T))}^2}{\alpha^2}
\right]
\|r^{(k)}\|_{L^2(0,T)}^2.
\]
Setting
\[
\widetilde{q}^2
:=
1-\frac{2\widetilde{\gamma}}{\alpha}
+
\frac{L_\phi^2
\|F\|_{\mathcal{L}(L^2(0,T))}^2}{\alpha^2},
\]
condition~\eqref{eq:alpha_convergence_condition} gives $\widetilde{q}<1$.
Consequently,
\[
\|r^{(k+1)}\|_{L^2(0,T)}
\le
\widetilde{q}\|r^{(k)}\|_{L^2(0,T)},
\]
and iteration yields~\eqref{eq:residual_geometric_decay}.

We now use the conditional H\"older stability estimate established in
Theorem~\ref{thm:conditional_stability}. Since
$f^{(k)}$ and $f^\dagger$ satisfy the prescribed $H^r$-bounds,
\[
\|f^{(k)}-f^\dagger\|_{L^2(0,T)}
\le
\widetilde{C}
\|F(f^{(k)})-F(f^\dagger)\|_{L^2(0,T)}^{\frac{r}{r+1}} = \widetilde{C}
\|r^{(k)}\|_{L^2(0,T)}^{\frac{r}{r+1}}.
\]
Using~\eqref{eq:residual_geometric_decay}, we obtain
\[
\|f^{(k)}-f^\dagger\|_{L^2(0,T)}
\le
\widetilde{C}
\widetilde{q}^{k\frac{r}{r+1}}
\|r^{(0)}\|_{L^2(0,T)}^{\frac{r}{r+1}},
\]
which proves~\eqref{eq:source_feedback_convergence}. Since $0\le \widetilde{q}<1$,
the right-hand side converges to zero as $k\to\infty$. Thus,
\[
f^{(k)}\to f^\dagger
\quad\text{in }L^2(0,T).
\]
\end{proof}

We now combine the exact-data convergence result with the
finite-iteration stability estimate to obtain an iterative
regularization result for noisy data.

\begin{theorem}[Iterative regularization by early stopping]
\label{thm:iterative_regularization}
Let $h=F(f^\dagger)$ be exact data and let $h^\delta$ satisfy
\[
\|h^\delta-h\|_{L^2(0,T)}\le\delta.
\]
Let $\{f^{(k)}\}_{k\ge0}$ and
$\{f_\delta^{(k)}\}_{k\ge0}$ denote the exact-data and noisy-data
feedback iterates, respectively, generated from the same initial
guess, $f_\delta^{(0)}=f^{(0)}.$

Assume that the hypotheses of
Theorem~\ref{thm:exact_feedback_convergence} hold, so that for some
$\widetilde q\in[0,1)$,
\[
\|f^{(k)}-f^\dagger\|_{L^2(0,T)}
\le
{\widetilde C}\,
\widetilde q^{\,k\frac{r}{r+1}}
\|r^{(0)}\|_{L^2(0,T)}^{\frac{r}{r+1}}.
\]
Define
\[
q
:=
1+
\frac{L_\phi
\|F\|_{\mathcal{L}(L^2(0,T))}}{\alpha}.
\]

Let $k_\delta\in\mathbb{N}$ be a noise-dependent stopping index
satisfying
\begin{equation}\label{eq:stopping_conditions}
k_\delta\to\infty
\qquad\text{and}\qquad
\delta\,q^{k_\delta}\to0
\quad\text{as }\delta\to0.
\end{equation}
Then
\[
f_\delta^{(k_\delta)}
\longrightarrow
f^\dagger
\qquad\text{in }L^2(0,T)
\quad\text{as }\delta\to0.
\]
More precisely,
\begin{equation}\label{eq:iterative_regularization_early_stopping_bound}
\|f_\delta^{(k_\delta)}-f^\dagger\|_{L^2(0,T)}
\le
\frac{\delta}
{\|F\|_{\mathcal{L}(L^2(0,T))}}
\left(
q^{k_\delta}-1
\right)
+
{\widetilde C}\,
\widetilde q^{\,k_\delta\frac{r}{r+1}}
\|r^{(0)}\|_{L^2(0,T)}^{\frac{r}{r+1}}.
\end{equation}
\end{theorem}

\begin{proof}
By the triangle inequality,
\begin{align*}
\|f_\delta^{(k_\delta)}-f^\dagger\|_{L^2(0,T)}
&\le
\|f_\delta^{(k_\delta)}-f^{(k_\delta)}\|_{L^2(0,T)}
+
\|f^{(k_\delta)}-f^\dagger\|_{L^2(0,T)}.
\end{align*}

By Theorem~\ref{thm:finite_iteration_stability},
\[
\|f_\delta^{(k_\delta)}-f^{(k_\delta)}\|_{L^2(0,T)}
\le
\frac{\delta}
{\|F\|_{\mathcal{L}(L^2(0,T))}}
\left(
q^{k_\delta}-1
\right).
\]
On the other hand, Theorem~\ref{thm:exact_feedback_convergence}
gives
\[
\|f^{(k_\delta)}-f^\dagger\|_{L^2(0,T)}
\le
{\widetilde C}\,
\widetilde q^{\,k_\delta\frac{r}{r+1}}
\|r^{(0)}\|_{L^2(0,T)}^{\frac{r}{r+1}}.
\]
Combining these two estimates yields
\eqref{eq:iterative_regularization_early_stopping_bound}.
By~\eqref{eq:stopping_conditions},
\[
\delta
\left(
q^{k_\delta}-1
\right)
\longrightarrow0
\qquad\text{as }\delta\to0.
\]
Moreover, since $k_\delta\to\infty$ and
$0\le\widetilde q<1$, $\widetilde q^{\,k_\delta\frac{r}{r+1}}\longrightarrow0$ as $\delta\to0.$
Therefore both terms on the right-hand side of
\eqref{eq:iterative_regularization_early_stopping_bound} converge to zero, and hence
\[
\|f_\delta^{(k_\delta)}-f^\dagger\|_{L^2(0,T)}
\longrightarrow0
\qquad\text{as }\delta\to0.
\]
This completes the proof.
\end{proof}

\begin{remark}
The stopping conditions in~\eqref{eq:stopping_conditions} are compatible.
Indeed, for any $\vartheta\in(0,1)$, the choice
\[
k_\delta
=
\left\lfloor
\frac{\vartheta\log(1/\delta)}
{\log q}
\right\rfloor
\]
satisfies $k_\delta\to\infty$ as $\delta\to0$. Moreover, $q^{k_\delta}\le \delta^{-\vartheta},$
and hence $\delta q^{k_\delta}\le \delta^{1-\vartheta}\to0.$
Thus, admissible noise-dependent stopping indices satisfying
\eqref{eq:stopping_conditions} exist.
\end{remark}

\section{Numerical Solutions}\label{sec:numerical_results}
\subsection{Numerical Implementation} 
In this section, we present numerical simulations to illustrate our theoretical results. 

All numerical experiments are performed on the spatial domain $\Omega = [0,1]^2$ using a uniform grid with $n_x = n_y = 41$ points in each spatial direction. The time interval $[0,T]$ with $T = 1$ is discretized uniformly with $N_t = 400$ time steps, and the step size $\Delta t = T/N_t$.

All simulations are implemented in Python using standard numerical linear algebra routines.

Two test cases are examined. In the first case, we recover a smooth time-dependent source $f(t)$ from a problem admitting an analytic solution. In the second case, we reconstruct a piecewise constant source exhibiting jump discontinuities.

In both numerical examples, the known distributed source and boundary data vanish,
$\zeta=0$ and $g=0$. Although the initial condition is nonzero, its contribution to
the integral observation has zero spatial mean. Consequently,
\[
\int_\Omega u_{u_0}(t,x)\,dx=0,
\]
and hence the adjusted data discussed in Section~\ref{sec:feedback_analysis}
satisfy 
\[
h^\delta(t)=\mu^\delta(t).
\]
Accordingly, we use the measured data $\mu^\delta$ directly in the numerical
feedback iteration.

The forward problem is solved using a backward Euler scheme in time combined with a finite difference discretization in space with homogeneous Neumann boundary conditions.

For the inverse iteration, we employ the adjoint-free integral feedback method as follows.

\subsection{Temporal regularization of the feedback iteration}

We now describe the numerical implementation of the adjoint-free integral
feedback method analyzed in Section~\ref{sec:feedback_analysis}. Although early stopping provides an iterative regularization mechanism, as established in
Theorem~\ref{thm:iterative_regularization}, the presence of noisy data may still introduce oscillations in the reconstructed source. To further stabilize the numerical reconstruction, we supplement the feedback iteration with an explicit temporal smoothing step.

Given the feedback update
\[
g^{(k)}(t)
:=
f^{(k)}(t)
-
\frac{1}{\alpha}\phi\bigl(r^{(k)}(t)\bigr),
\]
we define the regularized iterate $f^{(k+1)}$ by
\[
\left(I+\lambda_{\mathrm{reg}}\mathcal{R}\right)f^{(k+1)}
=
g^{(k)},
\]
where $\lambda_{\mathrm{reg}}>0$ is the regularization parameter and
$\mathcal{R}$ is a temporal smoothing operator. In the discrete
implementation, the regularization is based on the second-difference
operator $D_2$, leading to the smoothing operator $D_2^\top D_2$.
This additional regularization suppresses temporal oscillations while the
feedback correction remains driven by the data misfit.

\subsection{Discrete integral feedback algorithm}

Let $0=t_0<t_1<\dots<t_{N_t}=T$ be a uniform partition with $\Delta t = T/N_t$.
We represent the unknown source by the time-sampled vector
\[
\mathbf f^{(k)} = \bigl(f^{(k)}(t_0),\dots,f^{(k)}(t_{N_t})\bigr)^\top.
\]
Given $\mathbf f^{(k)}$, we compute the corresponding state $u^{(k)}$ by a time-stepping
scheme and then update the source by an integral feedback law. In our implementation, we discretize the model problem
\[
u_t - \Delta u + u = f(t)
\]
with homogeneous Neumann boundary conditions using a backward Euler scheme in time and a
finite difference discretization in space. At each time step $t_n$, this leads to a linear system
of the form
\[
A\,\mathbf u^{(k)}_n = \mathbf u^{(k)}_{n-1} + \Delta t\, f^{(k)}_n \mathbf{1},
\]
where $A$ is the discrete elliptic operator, $\mathbf{1}$ denotes the vector of ones, and $\mathbf u^{(k)}_{n}$ is the solution over space (state vector) at time $t_n$ for iteration $k$.

The integral observation is approximated by
\[
F^{(k)}_n := \int_\Omega u^{(k)}(x,t_n)\,dx,
\]
which, in the discrete setting, corresponds to a weighted sum over the spatial grid.

The residual is defined as
\[
r^{(k)}_n := F^{(k)}_n - \mu^\delta_n,
\qquad n=0,\dots,N_t,
\]
where $\mu^\delta_n := \mu^\delta(t_n)$ denotes the sampled noisy data. 

The discrete integral feedback update (prior to regularization) is given by
\[
g^{(k)}_n := f^{(k)}_n - \frac{1}{\alpha}\,
\phi\bigl(r^{(k)}_n\bigr),
\qquad n=0,\dots,N_t.
\]
We denote the corresponding feedback vector by
\[
\mathbf g^{(k)}
:=
\bigl(g^{(k)}_0,\dots,g^{(k)}_{N_t}\bigr)^\top.
\]

Let
$D_2\in\mathbb{R}^{(N_t-1)\times(N_t+1)}$ denote the
second-difference operator defined by
\[
(D_2 f)_i = f_i - 2f_{i+1} + f_{i+2},
\qquad i=0,\dots,N_t-2.
\]

The regularized iterate $\mathbf f^{(k+1)}$ is then defined as the
solution of
\[
\left(I+\lambda_{\mathrm{reg}}D_2^\top D_2\right)
\mathbf f^{(k+1)}
=
\mathbf g^{(k)}.
\]
This penalizes the temporal curvature and suppresses oscillations in
the reconstructed source.

\subsection{Noise model}\label{subsec:noise_model}
We define the discrete noiseless data by $\mu_n:=\mu(t_n)$. To simulate measurement noise, we consider additive Gaussian perturbations.
Let $(\eta_n)_{n=0}^{N_t}$ be independent standard normal random variables. Given a prescribed
relative noise level $\delta\in(0,1)$, we define
\[
\mu^\delta = \mu + \delta\,\|\mu\|_{L^2_\Delta(0,T)}\,\frac{\eta}{\|\eta\|_{L^2_\Delta(0,T)}},
\]
where the discrete norm is defined by
\begin{equation}\label{eq:discrete_L2}
\|v\|_{L^2_\Delta(0,T)} :=
\left(\Delta t\sum_{n=0}^{N_t} |v_n|^2\right)^{1/2}.
\end{equation}
By construction, this yields
\[
\frac{\|\mu^\delta-\mu\|_{L^2_\Delta(0,T)}}{\|\mu\|_{L^2_\Delta(0,T)}}=\delta.
\]

\subsection{Stopping rule (discrepancy principle)}\label{subsec:stopping_rule}
We measure the size of the residual $r^{{(k)}}$ using the discrete $L^2$-norm~\eqref{eq:discrete_L2}.
Let
\[
\delta_\Delta := \|\mu^\delta-\mu\|_{L^2_\Delta(0,T)}.
\]
We terminate the iteration by the discrepancy principle: for a fixed $\tau>1$, we stop at the first
index $k_*$ such that
\[
\|r^{(k_*)}\|_{L^2_\Delta(0,T)} \le \tau\,\delta_\Delta.
\]

\medskip
\begin{algorithm}[H]
\caption{Regularized Integral Feedback Method (adjoint-free)}
\label{alg:IFM}
\DontPrintSemicolon

\KwIn{
Time grid $\{t_n\}$, noisy data $\{\mu^\delta_n\}$, initial guess $\{f^{(0)}_n\}$,
parameters $\alpha>0$, $\lambda_{\mathrm{reg}}\ge 0$, tolerance $\tau>1$,
noise level $\delta_\Delta$, maximum iterations $k_{\max}$.
}

\For{$k=0,\dots,k_{\max}-1$}{
Solve the forward problem to obtain $u^{(k)}$\;

Compute $F^{(k)}_n = \int_\Omega u^{(k)}(x,t_n)\,dx$\;

Set $r^{(k)}_n = F^{(k)}_n - \mu^\delta_n$\;

\If{$\|r^{(k)}\|_{L^2_\Delta} \le \tau\,\delta_\Delta$}{
break\;
}

Compute feedback update prior to regularization
\[
g^{(k)}_n = f^{(k)}_n - \frac{1}{\alpha}\phi(r^{(k)}_n)
\]

\If{$\lambda_{\mathrm{reg}}>0$}{
Solve
\[
(I+\lambda_{\mathrm{reg}}D_2^\top D_2)\mathbf f^{(k+1)} = \mathbf g^{(k)}
\]
}
\Else{
$f^{(k+1)} = g^{(k)}$
}
}
\end{algorithm}

\subsection{Error metrics and reported quantities}\label{subsec:error_metrics}

We quantify the smoothness of the recovered source via the curvature seminorm
\[
\|D_2 f\|_{\ell^2}
=
\left(\sum_{i=0}^{N_t-2} |f_i - 2f_{i+1} + f_{i+2}|^2\right)^{1/2}.
\]
This quantity measures oscillations in time and provides an indicator of the regularization effect. We report the relative source error
\[
\mathrm{RelErr}_f^{(k)} :=
\frac{\|f^{(k)}-f^\dagger\|_{L^2_\Delta(0,T)}}{\|f^\dagger\|_{L^2_\Delta(0,T)}},
\]
together with the curvature and the stopping index $k_*$. Here, $f^\dagger$ is the true source function.

\subsection{Parameter selection}

The performance of the method depends on the choice of the feedback parameter $\alpha$
and the regularization parameter $\lambda_{\mathrm{reg}}$. In our experiments, these parameters are selected automatically over discrete grids. For each noise level, we consider
\[
\alpha \in \{3,5,7,10\}, \qquad
\lambda_{\mathrm{reg}} \in \{10^{-5},\dots,10\}.
\]

Among all reconstructions satisfying the discrepancy principle, we select the one with minimal curvature
\[
\|D_2 f^{(k_*)}\|_{\ell^2}.
\]
This strategy is designed to balance data fidelity and smoothness.

We use the discrepancy parameter $\tau = 1.05$, the maximum number of iterations $k_{\max} = 600$, and the initial guess $f^{(0)} \equiv 0$. 

We employ the nonlinear feedback function 
\[
\phi(s)=\tanh(s), 
\]
which provides a natural saturation mechanism by damping large residuals, since $\tanh(s)\to\pm1$ as $|s|\to\infty$, while preserving sensitivity to small discrepancies, since $\tanh(s)\approx s$ for small $|s|$. This bounded nonlinear response helps prevent excessively large feedback corrections in the presence of noisy data.

\begin{remark}
The feedback function used in the numerical experiments,
$\phi(s)=\tanh(s)$, satisfies
\[
\phi(0)=0,
\qquad
|\phi'(s)|=\operatorname{sech}^2(s)\le 1,
\]
and is therefore globally Lipschitz with $L_\phi=1$. Moreover,
\[
\tanh(s)=0
\quad\Longleftrightarrow\quad
s=0.
\]
Hence, the assumptions on the feedback function in Theorems~\ref{thm:feedback_well_defined}
and~\ref{thm:feedback_fixed_points}
are satisfied by the feedback function employed in the numerical simulations.
\end{remark}

\subsection{A problem with analytic solution}
\subsubsection{Problem setup}
In this section, we approximate the solution of the parabolic problem~\eqref{parab2}-\eqref{bc2} subject to~\eqref{forward_map2}.
For simplicity, we choose constant coefficients
\[
a(t,x)=1,\qquad b(t,x)\equiv 0,\qquad c(t,x)=1.
\]
Other configurations, such as vanishing convection or reaction terms, could also be considered,
but diffusion is retained in order to preserve the parabolic structure of the model.

We first construct an analytic solution
to this inverse problem in order to generate synthetic data, and then apply the proposed integral feedback algorithm to recover the unknown source.

To this end, consider
\[
\partial_t u-\Delta u + u = f(t)
\quad\text{in }(0,T)\times\Omega,\qquad
\partial_\nu u=0\ \text{on }(0,T)\times\partial\Omega.
\]
Let $\psi(x,y)=\cos(\pi x)\cos(\pi y)$. Then $\partial_\nu\psi=0$ on $\partial\Omega$ and
$-\Delta\psi=\lambda\psi$ with $\lambda=2\pi^2$. Moreover,
\[
\int_\Omega \psi(x,y)\,dx\,dy
=\Bigl(\int_0^1 \cos(\pi x)\,dx\Bigr)\Bigl(\int_0^1 \cos(\pi y)\,dy\Bigr)=0.
\]
We seek a solution of the form
\[
u(t,x,y)=m(t)+w(t)\psi(x,y).
\]
Substituting into the PDE yields
\[
m'(t)+m(t)=f(t),\qquad w'(t)+(\lambda+1)w(t)=0.
\]
Choosing $f(t)=\sin(2\pi t)$, $m(0)=0$, and $w(0)=1$, we obtain
\[
m(t)=\frac{\sin(2\pi t)-2\pi\cos(2\pi t)+2\pi e^{-t}}{1+4\pi^2},
\qquad
w(t)=e^{-(2\pi^2+1)t}.
\]
Hence an explicit spatially dependent solution is
\[
u(t,x,y)=
\frac{\sin(2\pi t)-2\pi\cos(2\pi t)+2\pi e^{-t}}{1+4\pi^2}
+
e^{-(2\pi^2+1)t}\cos(\pi x)\cos(\pi y).
\]
Finally, the corresponding integral observation satisfies
\begin{equation}\label{eq:observed_data}
(Ff)(t)=\int_\Omega u(t,x,y)\,dx\,dy
=|\Omega|\,m(t)
=
\frac{\sin(2\pi t)-2\pi\cos(2\pi t)+2\pi e^{-t}}{1+4\pi^2}.
\end{equation}

\subsubsection{Results}
Figure~\ref{fig:unregularized_sources} shows the true source $f(t)$ and the noisy reconstructions with $1\%$, $3\%$, and $5\%$ noise added and no regularization. 
\begin{figure}[htbp]
\centering
\includegraphics[width =\textwidth]{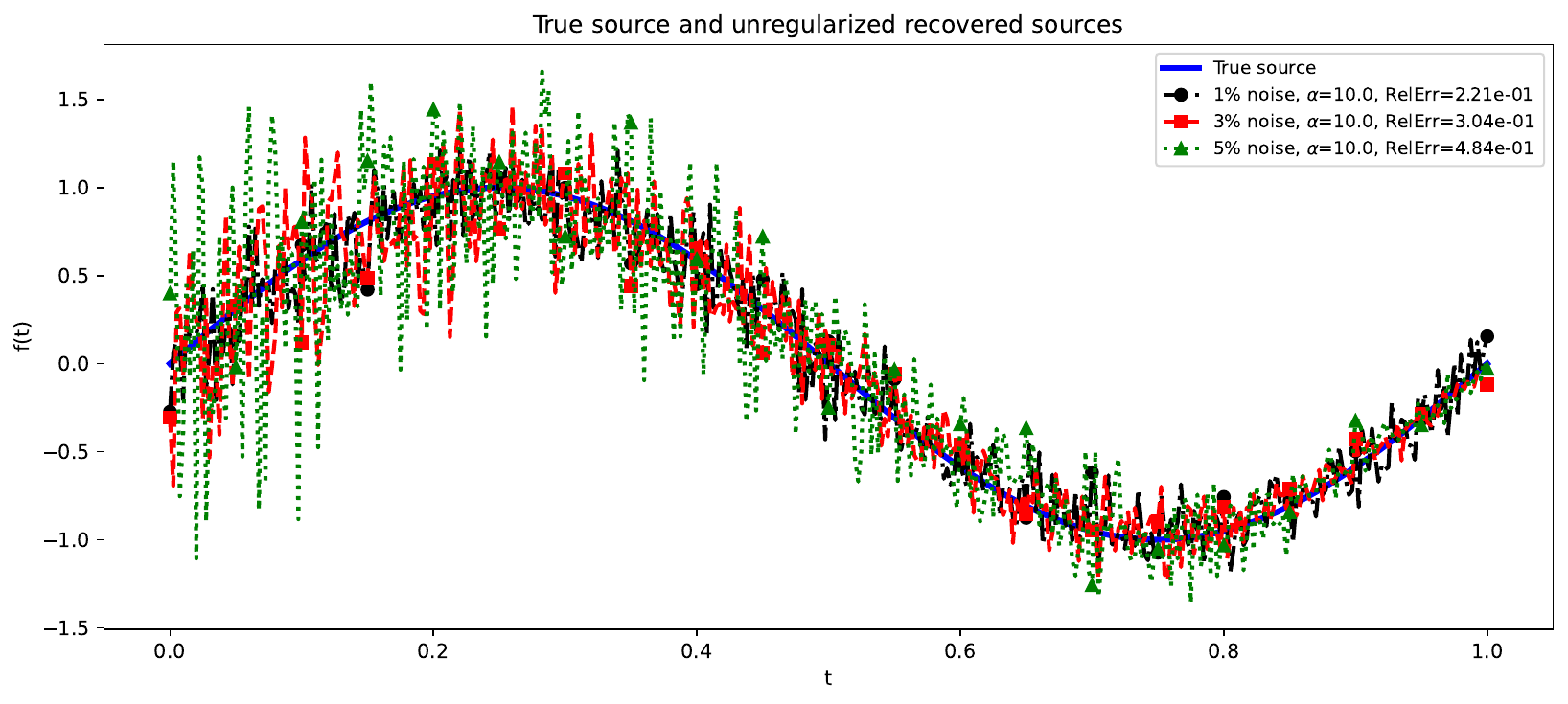}
\caption{True analytic source and unregularized reconstructions with $1\%$, $3\%$, and $5\%$ noise.}
\label{fig:unregularized_sources}
\end{figure}

\newpage
On the other hand, in Figure~\ref{fig:regularized_sources} we plot the true source and regularized sources corresponding to the noisy unregularized sources in Figure~\ref{fig:unregularized_sources}.
\begin{figure}[hbtp]
\centering
\includegraphics[width =\textwidth]{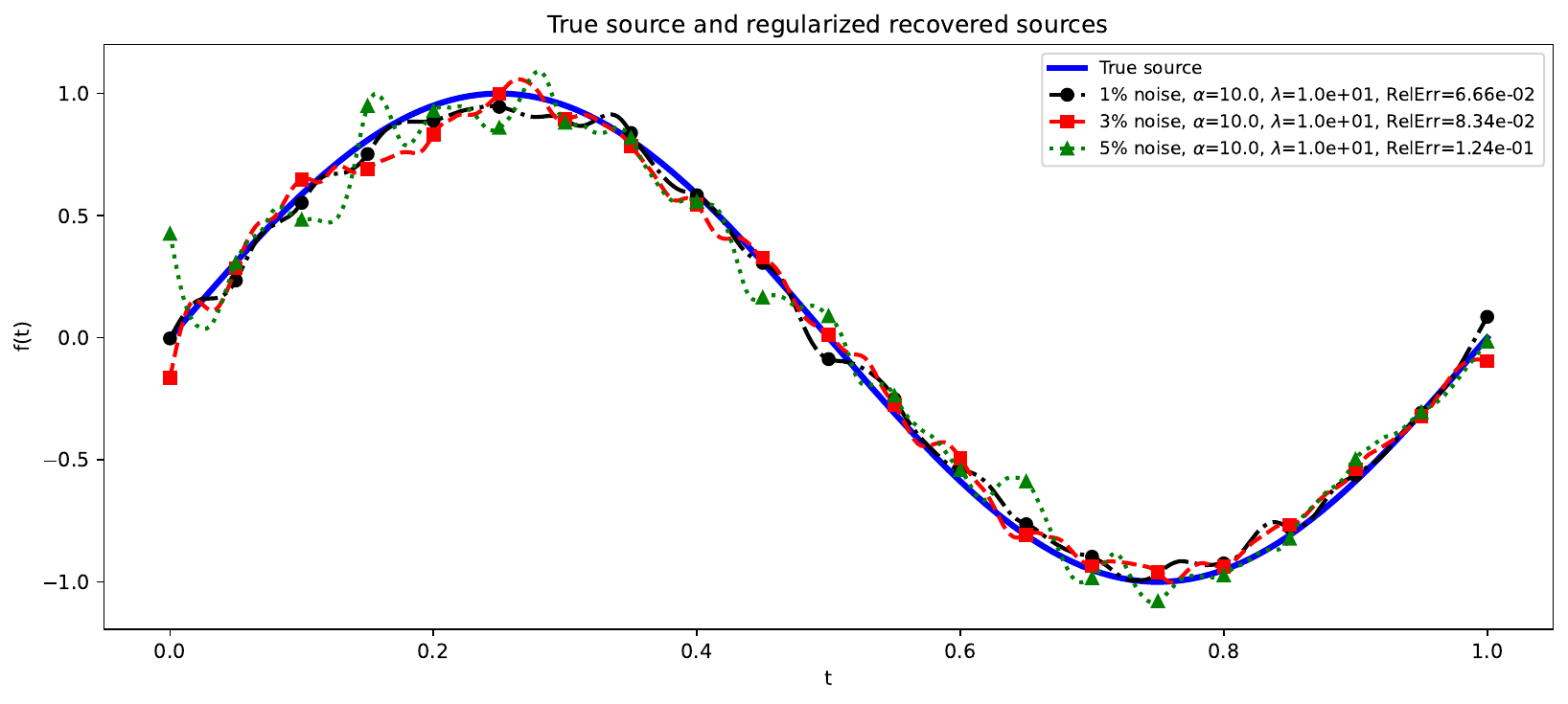}
\caption{True analytic source and regularized reconstructions with $1\%$, $3\%$, and $5\%$ noise}
\label{fig:regularized_sources}
\end{figure}

We note that the chosen source $f(t)=\sin(2\pi t)$ changes sign on $[0,T]$, which allows us to assess the ability of the method to reconstruct both positive and negative source contributions.

Table~\ref{tab:smooth_source_results} provides the values of the feedback parameter $\alpha$, the regularization parameter $\lambda_{\mathrm{reg}}$, the stopping index $k_\ast$, the relative error and curvature for the reconstructions.

\begin{table}[H]
\centering
\caption{Reconstruction results for the smooth source using parameter selection over $(\alpha,\lambda_{\mathrm{reg}})$.}
\label{tab:smooth_source_results}
\begin{tabular}{c c c c c c}
\hline
Noise & $\alpha$ & $\lambda_{\mathrm{reg}}$ & $k_\ast$ & RelErr$_f$ & Curvature \\ 
\hline
\multicolumn{6}{c}{\textbf{Unregularized reconstructions}} \\
\hline
1\% & 10.000 & $0$ & 40 & $2.21\times 10^{-1}$ & $8.01$ \\
3\% & 10.000 & $0$ & 23 & $3.04\times 10^{-1}$ & $9.45$ \\
5\% & 10.000 & $0$ & 23 & $4.84\times 10^{-1}$ & $16.39$ \\
\hline
\multicolumn{6}{c}{\textbf{Regularized reconstructions (automatic $(\alpha,\lambda)$ selection)}} \\ 
\hline
1\% & 10.000 & $10.0$ & 54 & $6.66\times 10^{-2}$ & $3.52\times 10^{-2}$ \\
3\% & 10.000 & $10.0$ & 24 & $8.34\times 10^{-2}$ & $6.66\times 10^{-2}$ \\
5\% & 10.000 & $10.0$ & 23 & $1.24\times 10^{-1}$ & $9.79\times 10^{-2}$ \\
\hline
\end{tabular}
\end{table}

We observe that, in both the regularized and unregularized cases, the stopping index $k_\ast$ generally increases as the noise level decreases. This behavior is consistent with the iterative-regularization framework of Theorem~\ref{thm:iterative_regularization}, in which the number of iterations is allowed to increase as the noise level tends to zero.

The results demonstrate a substantial improvement in reconstruction accuracy when both the feedback parameter $\alpha$ and the regularization parameter $\lambda_{\mathrm{reg}}$ are selected automatically. In particular, the relative $L^2(0,T)$ error decreases significantly across all noise levels.

Moreover, the curvature of the reconstructed source is drastically reduced, indicating that the second-order temporal regularization effectively suppresses oscillations. Interestingly, the parameter selection procedure consistently selects $\alpha=10$ and $\lambda_{\mathrm{reg}}=10$ across all noise levels, suggesting that stronger damping in the feedback iteration combined with stronger smoothing yields robust and accurate reconstructions.

\subsection{A problem with no analytic solution: recovering a piecewise constant source with jump discontinuities}

\subsubsection{Problem setup}
We now consider the recovery of a non-smooth source exhibiting jump discontinuities. This test case is designed to assess the robustness of the proposed method in the presence of reduced regularity. We define the true source function $f^\dagger:[0,T]\to\mathbb{R}$ by
\[
f^\dagger(t) =
\begin{cases}
1, & 0 \le t < 0.3, \\
-0.5, & 0.3 \le t < 0.7, \\
0.8, & 0.7 \le t \le T.
\end{cases}
\]

In contrast to the previous example, the inverse problem has no analytic solution. The corresponding observation data are therefore generated numerically by solving the forward problem and computing
\[
\mu(t) = \int_\Omega u^\dagger(x,t)\,dx,
\]
where $u^\dagger$ denotes the solution associated with the true source $f^\dagger$.

Noisy measurements $\mu^\delta$ are constructed as described in Section~\ref{subsec:noise_model}. All other aspects of the numerical setup are kept identical to those used in the smooth-source example, including the initial condition, homogeneous Neumann boundary conditions, discretization parameters, stopping criterion, and the automatic selection of the feedback and regularization parameters. This ensures a consistent and fair comparison between the smooth and non-smooth cases.

\subsubsection{Results}
Figure~\ref{fig:unregularized_piecewise} shows the true piecewise constant source and the noisy reconstructions with $1\%$, $3\%$, and $5\%$ noise added and no regularization. 

\begin{figure}[htbp]
\centering
\includegraphics[width =\textwidth]{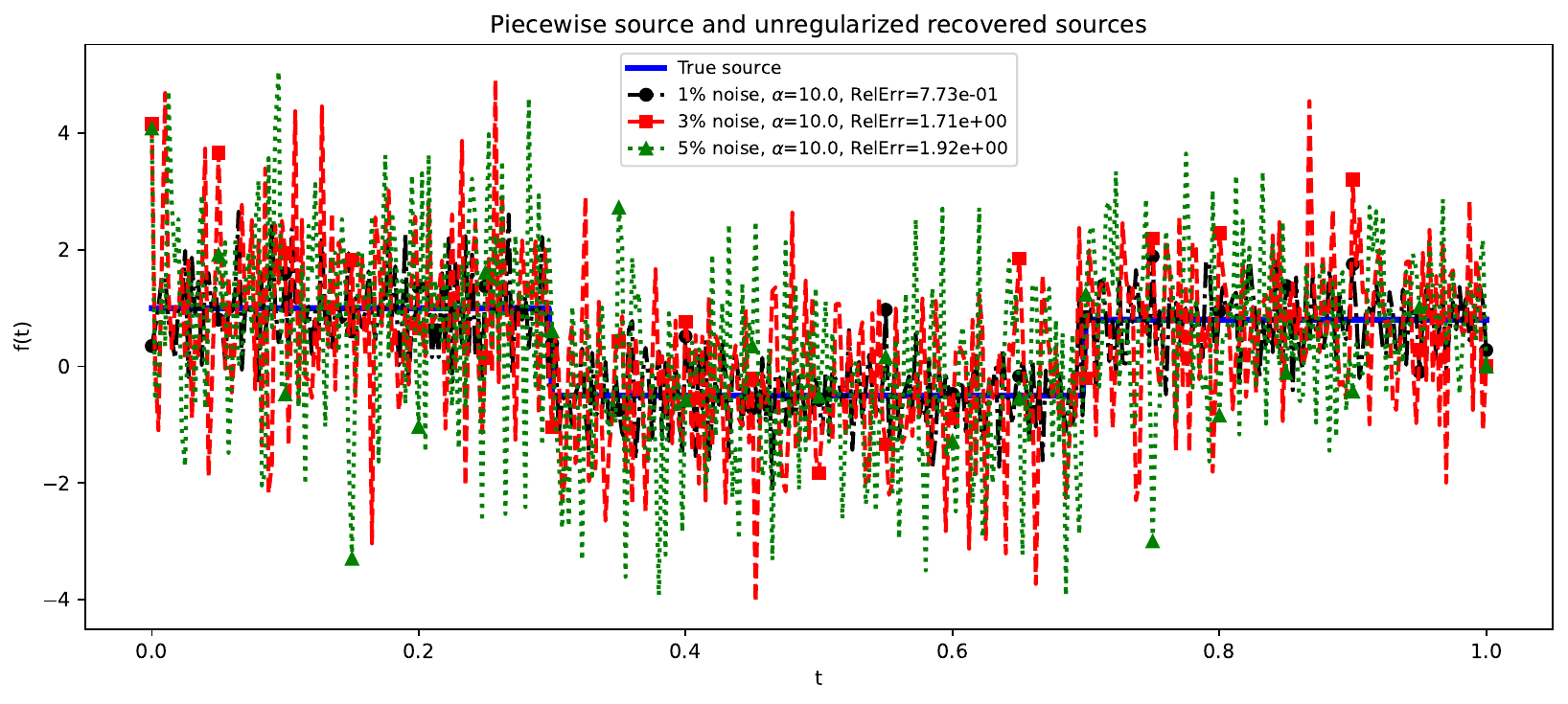}
\caption{True piecewise constant source and unregularized reconstructions for noise levels $1\%$, $3\%$, and $5\%$. The reconstructed sources exhibit strong oscillations and poor stability.}
\label{fig:unregularized_piecewise}
\end{figure}

On the other hand, in Figure~\ref{fig:regularized_piecewise} we plot the true piecewise constant source and regularized reconstructions corresponding to the noisy unregularized sources in Figure~\ref{fig:unregularized_piecewise}.
\begin{figure}[htbp]
\centering
\includegraphics[width =\textwidth]{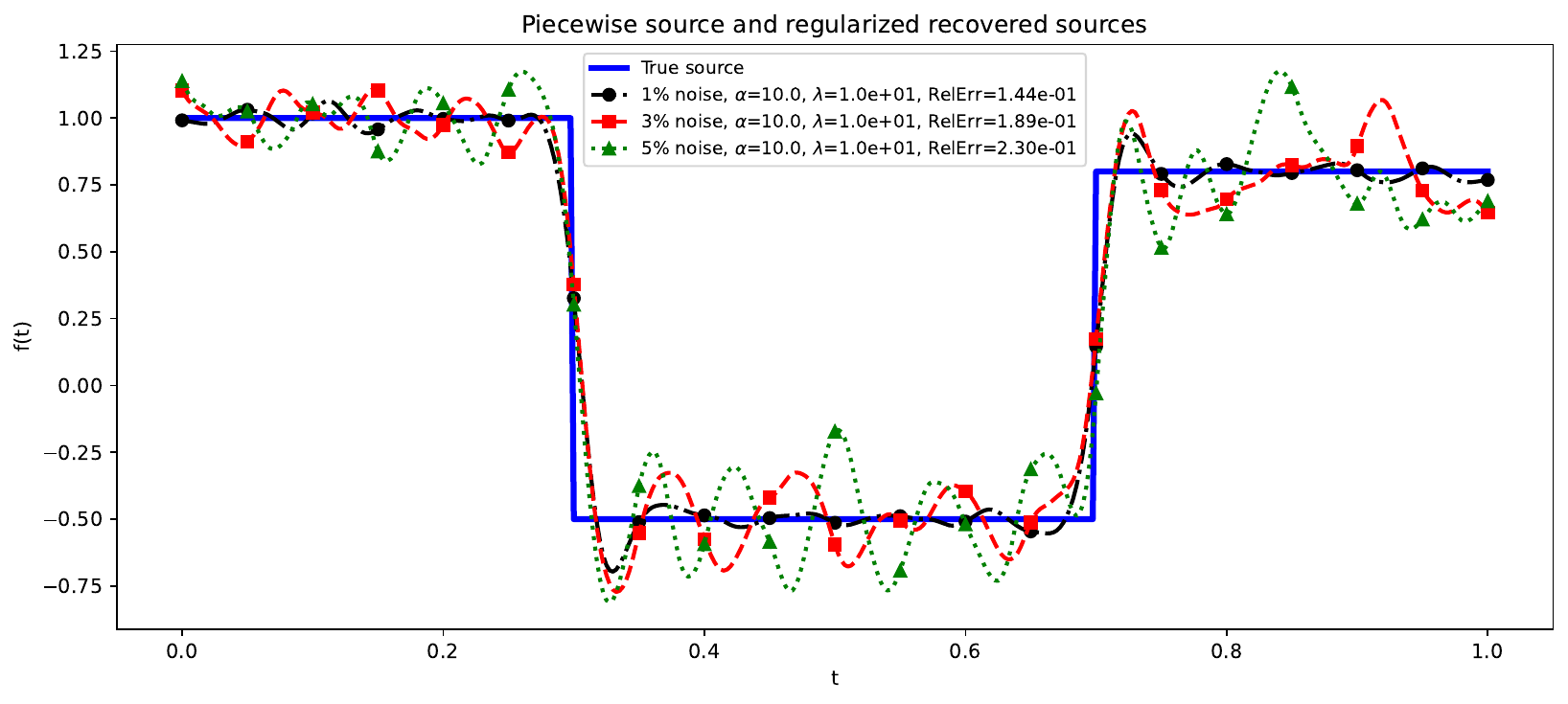}
\caption{True piecewise constant source and regularized reconstructions for noise levels $1\%$, $3\%$, and $5\%$. The regularized integral feedback method captures the main structure of the source.}
\label{fig:regularized_piecewise}
\end{figure}

Table~\ref{tab:piecewise_source_results} provides the values of the feedback parameter $\alpha$, the regularization parameter $\lambda_{\mathrm{reg}}$, the stopping index $k_\ast$, the relative error and curvature for the reconstructions.

\begin{table}[H]
\centering
\caption{Reconstruction results for the piecewise constant source example using parameter selection over $(\alpha,\lambda_{\mathrm{reg}})$.}
\label{tab:piecewise_source_results}
\begin{tabular}{c c c c c c}
\hline
Noise & $\alpha$ & $\lambda_{\mathrm{reg}}$ & $k_\ast$ & RelErr$_f$ & Curvature \\
\hline
\multicolumn{6}{c}{\textbf{Unregularized reconstructions}} \\
\hline
1\% & 10.000 & $0$ & 202 & $7.73\times 10^{-1}$ & $34.04$ \\
3\% & 10.000 & $0$ & 137 & $1.71$ & $67.77$ \\
5\% & 10.000 & $0$ & 93  & $1.92$ & $73.29$ \\
\hline
\multicolumn{6}{c}{\textbf{Regularized reconstructions (automatic $(\alpha,\lambda)$ selection)}} \\
\hline
1\% & 10.000 & $10.0$ & 405 & $1.44\times 10^{-1}$ & $8.50\times 10^{-2}$ \\
3\% & 10.000 & $10.0$ & 184 & $1.89\times 10^{-1}$ & $1.19\times 10^{-1}$ \\
5\% & 10.000 & $10.0$ & 150 & $2.30\times 10^{-1}$ & $1.88\times 10^{-1}$ \\
\hline
\end{tabular}
\end{table}

The same trend is observed for the piecewise constant source: the stopping index generally increases as the noise level decreases, again consistent with the iterative-regularization framework of Theorem~\ref{thm:iterative_regularization}.

The piecewise constant source example provides a significantly more challenging test for the proposed method due to the presence of jump discontinuities and reduced regularity of the unknown source. As expected, the unregularized reconstructions are highly oscillatory and unstable, particularly as the noise level increases. In contrast, the regularized integral feedback method produces stable and accurate reconstructions across all noise levels.

In particular, the relative $L^2(0,T)$ error decreases substantially for all noise levels, while the curvature of the recovered source is drastically reduced. This demonstrates that the second-order temporal regularization effectively suppresses oscillations while preserving the main structure of the piecewise constant source.

These results confirm that the proposed adjoint-free integral feedback method remains robust even when the true source is discontinuous, provided that an appropriate temporal regularization is incorporated.

\begin{remark}
For piecewise constant sources, a first-order temporal regularization is more naturally aligned with the structure of the unknown. We tested this approach and observed comparable performance, with slight improvements at higher noise levels. However, to maintain a unified framework across both smooth and nonsmooth cases, we retain the second-order regularization throughout the paper.
\end{remark}

\section{Conclusion and Future Work}
In this work, we studied an inverse source problem for a linear non-autonomous
parabolic equation in which the unknown source depends only on time and is
recovered from an integral observation of the solution. After reducing the
original problem to an equivalent linear inverse problem, we established the
existence and uniqueness of the Tikhonov-regularized solution and derived its
first-order optimality condition. We also showed that the forward operator
admits a Volterra representation in time. This structure yields a weak-norm
stability estimate in $H^{-1}(0,T)$ and, under an a priori $H^r(0,T)$ bound on
the source, a conditional H\"older stability estimate in $L^2(0,T)$.

A central contribution of the work is the mathematical analysis of the proposed
adjoint-free integral feedback method. By formulating the reconstruction
procedure as a nonlinear fixed-point iteration, we established its
well-definedness and Lipschitz continuity, characterized exact solutions as
fixed points, and analyzed the behavior of both exact- and noisy-data
iterations. In particular, we obtained convergence of the exact-data iteration
under suitable assumptions and a finite-iteration stability estimate with
respect to data perturbations. Combining these results, we showed that the
feedback iteration equipped with an appropriate noise-dependent stopping rule
constitutes an iterative regularization method.

The theoretical analysis is complemented by numerical experiments for smooth
and piecewise constant sources. To further control oscillations caused by noisy
data, the numerical implementation supplements early stopping with second-order
temporal regularization and automatic parameter selection. The experiments
demonstrate stable and accurate reconstructions across the tested noise levels
and show that the method remains effective for sources with jump
discontinuities.

Future work includes extending the integral feedback approach to
space-dependent inverse source problems, investigating alternative
regularization strategies adapted to different source structures, and
developing sharper convergence results and convergence-rate estimates for the
feedback iteration under suitable source conditions.

\bibliographystyle{plain}
\bibliography{Sample}

\end{document}